\documentclass[11pt]{article}
\title{\textbf{\Large Stefan Problems for Reflected SPDEs Driven by Space-Time White Noise}}
\author{Ben Hambly\footnote{ \href{mailto:ben.hambly@maths.ox.ac.uk}{ben.hambly@maths.ox.ac.uk}}\hspace{2mm} and Jasdeep Kalsi \footnote{\href{mailto:jasdeep.kalsi@maths.ox.ac.uk}{jasdeep.kalsi@maths.ox.ac.uk}}.\\\\
Mathematical Institute, University of Oxford\\\\}
\date{\today}

\usepackage{amsmath, amssymb, amsthm,bbm}
\usepackage{mathrsfs}
\usepackage{enumerate}
\newtheorem{thm}{Theorem}[section]

\newtheorem{cor}[thm]{Corollary}
\newtheorem{defn}[thm]{Definition}

\newtheorem{prop}[thm]{Proposition}
\newtheorem{rem}[thm]{Remark}

\usepackage[margin=1in]{geometry}
\usepackage{graphicx}
\usepackage{epstopdf}
\usepackage{hyperref}
\usepackage{amsmath}
\usepackage{amsfonts}
\usepackage{setspace}
\usepackage{subcaption}
\newlength{\bibitemsep}
\newlength{\bibparskip}
\let\oldthebibliography\thebibliography
\renewcommand\thebibliography[1]{%
  \oldthebibliography{#1}%
  \setlength{\parskip}{\bibitemsep}%
  \setlength{\itemsep}{\bibparskip}%
}
\usepackage{titlesec}
\titleformat{\section}[block]{\sffamily\Large\bfseries\filcenter}{\thesection}{1em}{}
\usepackage{fancyhdr}
\usepackage{indentfirst}
\usepackage{eqnarray}
\usepackage{enumerate}
\usepackage{commath}
\usepackage{amssymb}
\usepackage{afterpage}
\usepackage{bbm}
\usepackage{algorithm}
\usepackage{algpseudocode}
\usepackage{float}

\numberwithin{equation}{section}
\makeatletter

\newcommand{\xRightarrow}[2][]{\ext@arrow 0359\Rightarrowfill@{#1}{#2}}
\makeatother
\allowdisplaybreaks[1]
\begin{document}
\maketitle
\begin{abstract}
\noindent \footnotesize{We prove the existence and uniqueness of solutions to a one-dimensional Stefan Problem for reflected 
SPDEs which are driven by space-time white noise. The solutions are shown to exist until almost surely positive blow-up times. 
Such equations can model the evolution of phases driven by competition at an interface, with the dynamics of the shared boundary 
depending on the derivatives of two competing profiles at this point. The novel features here are the presence of space-time 
white noise; the reflection measures, which maintain positivity for the competing profiles; and a sufficient 
condition to make sense of the Stefan condition at the boundary. We illustrate the behaviour of the solution numerically to show that this
sufficient condition is close to necessary.}
\end{abstract}

\section{Introduction}

Stefan problems have been extensively studied since the original work by Josef Stefan in 1888, and have a number of applications 
in physics, engineering, biology and finance. Broadly speaking, they describe situations where an interface moves with pressure or 
relative pressure due to competition from two types. In this paper, we will study stochastic, reflected versions of this problem in 
one-dimension, which take the form
\begin{equation}\label{movingboundaryoriginal}\begin{split}
& \frac{\partial u^1}{\partial t}= \Delta u^1 + f_1(p(t)-x,t, u^1(t,p(t)-\cdot))+ \sigma_1(p(t)-x,t,u^1(t,p(t)-\cdot))\dot{W} + \eta^1 \\ 
& \frac{\partial u^2}{\partial t}= \Delta u^2 + f_2(x-p(t),t,u^2(t, p(t)+ \cdot))+ \sigma_2(x-p(t),t,u^2(t, p(t)+ \cdot))\dot{W} + \eta^2,
\end{split}
\end{equation}
where, at any time $t \geq 0$, $u^1$ is supported on $[p(t)-1,p(t)]$, $u^2$ is supported on $[p(t),p(t)+1]$ and $u^1(t,p(t)-1) = u^1(t,p(t)) = u^2(t,p(t)) = u^2(t,p(t)+1)= 0$. The point $p(t)$ evolves according to the equation 
\begin{equation*}
p^{\prime}(t)= h\left(\frac{\partial u^1}{\partial x}(t,p(t)^-), \frac{\partial u^2}{\partial x}(t,p(t)^+) \right).
\end{equation*}
Here, $\dot{W}$ is a space-time white noise; $f_i$ and $\sigma_i$, for $i=1,2$, are the drift and volatility for each type, which are determined by the distance to the
moving boundary $p(t)$; and $(\eta^1,\eta^2)$ are reflection measures for the functions $u^1$ and $u^2$ 
respectively, keeping the profiles positive and satisfying the conditions 
\begin{enumerate}[(i)]
\item $\int_0^{\infty} \int_{\mathbb{R}} u^1(t,x) \;  \eta^1(\textrm{d}t,\textrm{d}x)=0,$ and 
\item $\int_0^{\infty} \int_{\mathbb{R}} u^2(t,x) \;  \eta^2(\textrm{d}t,\textrm{d}x)=0.$
\end{enumerate} 
The equation describes the evolution of two reflected SPDEs which share a moving boundary. The derivative of the moving boundary 
is then determined by $h$, a locally Lipschitz function of the spatial derivatives of the two SPDEs at the shared boundary. We note 
here that, in general, reflected SPDEs of this type will only be up to  $1/2$-H\"{o}lder continuous in space. We will therefore require the
functions $\sigma_i$ to satisfy suitable conditions which will ensure that the volatility decays at least linearly at the interface, which 
will be shown to be sufficient for the existence of a spatial derivative there. 

We recall the classical one-sided version of the Stefan problem, which takes the form
\begin{equation*}\begin{split}
 \frac{\partial u}{\partial t}= \Delta u, \; \; \; \; \; \; \; \;   p^{\prime}(t)= -\frac{\partial u}{\partial x}(t,p(t)^+).
\end{split}
\end{equation*}
The profile of $u$ at a given time $t$ has support in the set $[p(t), \infty)$, and we have the Dirichlet condition $u(t,p(t))=0$. 
Typically, we have that $u_0 \geq 0$, from which it follows that $u \geq 0$. The equation for $p^{\prime}$ then implies that the 
boundary recedes with ``pressure" from $u$. This can be thought of as a model for the melting of a block of ice when in contact 
with a body of water. The point $p(t)$ represents the interface between the water and ice, and $u$ the temperature profile of 
the body of water. 

Recently, stochastic perturbations of this classical problem have received significant attention. In \cite{KK}, existence and uniqueness 
for solutions to a Stefan problem where the two sides satify SPDEs driven by spatially coloured noise is proved. A particular case of the 
corresponding problem when the SPDE is driven by space-time white noise was then studied in \cite{Zheng}. It is assumed in \cite{Zheng} that the 
volatility, $\sigma(x)$ vanishes faster than $x^{3/2}$ as $x \downarrow 0$ i.e. as the moving interface is approached. More recent work 
on such problems include the models in \cite{Muller} and in \cite{KM}. In these papers, as in \cite{KK}, the two sides satisfy SPDEs 
driven by spatially coloured noise, as this makes it easier to establish the existence of the spatial derivative at the interface.
Motivated by modelling the limit order book in financial markets, the authors also include a Brownian noise term in the moving boundary.  

In addition to stochastic moving boundary problems, the study of reflected SPDEs has also attracted interest. Equations of the type 
\begin{equation*}\label{reflll}
\frac{\partial u}{\partial t} = \Delta u + f(x,u)+ \sigma(x,u) \dot{W} + \eta,
\end{equation*}
where $\dot{W}$ is space-time white noise and $\eta$ is a reflection measure, were initially studied in \cite{NP} in the case of constant volatility i.e. $\sigma \equiv 1$. Existence for the equation in the case when 
$\sigma= \sigma(x,u)$ satisfies Lipschitz and linear growth conditions in its second argument was then proved in \cite{DMP} using 
a penalization method. 
Uniqueness for varying volatility $\sigma= \sigma(x,u)$ on compact spatial domains was then shown in \cite{Xu}, with the authors 
achieving this by decoupling the obstacle and SPDE components of the problem.  

In this paper, we aim to take a first step towards studying reflected stochastic Stefan problems, by proving an existence and uniqueness 
result. The main condition required is that the volatility decays at least linearly close to the boundary, with the remaining conditions 
on the coefficients being mild. We also characterise the blow-up time for such equations, and demonstrate that it coincides with blow-up 
of one of the derivatives of the profiles at the shared boundary.

The results here extend the work in \cite{Zheng} by incorporating reflection and allowing for drift and volatility coefficients which 
depend on the spatial variable as well as the solution itself, and are Lipschitz in the solution. The condition on the decay of the 
volatility at the boundary is also relaxed, and required to be linear only, and we are able to choose a general Lipschitz function $h$ to determine the boundary evolution.  The presence of the reflection measure requires us to consider the parabolic obstacle problem, which we prove is a contraction in the space $C([0,T]; \mathscr{H})$, where $\mathscr{H}$ is the space of continuous functions on $(0,1)$ which vanish at the endpoints and have derivatives at $0$. It also requires us to produce stronger estimates on the heat kernel so that we can bound the $L^p(\Omega ; C([0,T] ; \mathscr{H}))$-norms of our processes. We are able to relax the condition on the volatility by formulating the equation in the frame relative to the shared boundary point $p(t)$. Doing so removes the need to consider effects from the boundary shift in our Picard argument, which appear in \cite{Zheng}, enabling us to obtain the stronger result. This paper can also be contrasted with \cite{Paper}. In \cite{Paper}, the volatility coefficients are not required to decay at the shared interface, at the cost of having the boundary derivative given by a function of the profiles in the space of continuous paths. In particular, the classic Stefan boundary condition does not fall within the framework of \cite{Paper}.

The outline for this paper is as follows. We begin in Section 2 by defining our notion of solutions, motivating this by some simple calculations.
Section 2 is then concluded by a statement of our main existence and uniqueness theorem. In Section 3 we turn our attention to the
deterministic obstacle problem which corresponds to our equations here, and prove a key result which will allow us to ensure that the 
reflection component will not prevent our SPDE solutions from having derivatives at the boundary. Section 4 is dedicated to establishing 
some key estimates for the main proof, with some of the technical details deferred to the appendix. In Section 5 we prove our main result,
presenting the arguments for existence and uniqueness for our problem. This is done by first truncating the problem suitably and then 
performing a Picard iteration, making use of the estimates from Sections 3 and 4. We conclude in Section 6 with simulations of 
the equations, illustrating the appearance of derivatives when the coefficients fall within our framework and briefly exploring scenarios where 
this is not the case.

\section{Notion of Solution and Statement of Main Theorem}

We begin this section by defining the space $\mathscr{H}$. The solutions to our Stefan problem will be $\mathscr{H}$-valued processes.

\begin{defn}
\begin{equation*}
\mathscr{H}:= \left\{ f \in C([0,1]) \; | \; f(0)=f(1)=0, f^{\prime}(0) \textrm{ exists} \right\}.
\end{equation*}
We equip $\mathscr{H}$ with the norm
\begin{equation*}
\|f \|_{\mathscr{H}}:= \sup\limits_{x \in (0,1]} \left| \frac{f(x)}{x} \right|.
\end{equation*}
\end{defn}

\begin{rem}
$(\mathscr{H}, \|\cdot \|_{\mathscr{H}})$ is a Banach space. In addition, we can characterise $\mathscr{H}$ as follows
\begin{equation*}
\mathscr{H}= \left\{ f \in C([0,1]) \; | \; \exists g \in C([0,1]) \textrm{ s.t. } f(x)= xg(x), \; g(1)=0  \right\}
\end{equation*}
Clearly, for any $f \in \mathscr{H}$, the function $g$ such that $f(x)=xg(x)$ is unique, and we have that 
\begin{equation*}
\|f\|_{\mathscr{H}}= \|g\|_{\infty}.
\end{equation*}
\end{rem}

\begin{defn}
Let $(\Omega, \mathscr{F}, \mathscr{F}_t, \mathbb{P})$ be a complete filtered probability space. Suppose that $\dot{W}$ is a space-time white noise defined on this space. Define for $t \geq 0$ and $A \in \mathscr{B}(\mathbb{R})$, $$W_t(A):= \dot{W}([0,t] \times A).$$ We say that $\dot{W}$ respects the filtration $\mathscr{F}_t$ if $(W_t(A))_{t \geq 0, A \in \mathscr{B}(\mathbb{R})}$ is an $\mathscr{F}_t$- martingale measure i.e.  if for every $A \in \mathscr{B}(\mathbb{R})$, $(W_t(A))_{t \geq 0}$ is an $\mathscr{F}_t$-martingale.
\end{defn}

We now derive our notion of solution, by formally multiplying our equations by test functions and integrating by parts. The following set up is as in \cite{Paper}, and we include the calculations here for completeness.

Let $(\Omega, \mathscr{F}, \mathscr{F}_t, \mathbb{P})$ be a complete filtered probability space, and $\dot{W}$ a space-time white noise which respects the filtration $\mathscr{F}_t$. Suppose that $(u^1,\eta^1, u^2, \eta^2, p)$ is an $\mathscr{F}_t$-adapted process solving (\ref{movingboundaryoriginal}). Then $p:\mathbb{R}^+ \times \Omega \mapsto \mathbb{R}$ is a $\mathscr{F}_t$-adapted process such that the paths of $p(t)$ are $C^1$ almost surely (note that, in particular, $p$ is $\mathscr{F}_t$-predictable).  Let $\varphi \in C_c^{\infty}([0,\infty) \times (0,1))$, and define the function $\phi$ by setting $\phi(t,x)= \varphi(t, p(t)-x)$. By multiplying the equation for $u^1$ in (\ref{movingboundaryoriginal}) by such a $\phi$ and integrating over space and time, interpreting the derivatives in the usual weak sense, we obtain that for $t \geq 0$,
\begin{equation}\label{change_var1}
\begin{split}
\int_{\mathbb{R}} u^1(t,x) \phi(t,x) \textrm{d}x= & \int_{\mathbb{R}} u^1(0,x) \phi(0,x) \textrm{d}x +  \int_0^t \int_{\mathbb{R}} u^1(s,x) \frac{\partial \phi}{\partial t}(s,x) \textrm{d}x\textrm{d}s \\ & + \int_0^t \int_{\mathbb{R}} u^1(s,x) \frac{\partial^2 \phi}{\partial x^2}(s,x) \textrm{d}x\textrm{d}s  \\ & + \int_0^t \int_{\mathbb{R}} f_1(p(s)-x,s,u^1(s,p(s)-\cdot)) \phi(s,x) \textrm{d}x\textrm{d}s \\ & + \int_0^t \int_{\mathbb{R}} \sigma_1(p(s)-x,s,u^1(s,p(s)-\cdot)) \phi(s,x)W(\textrm{d}x,\textrm{d}s)\\  & + \int_0^t \int_{\mathbb{R}} \phi(s,x) \; \eta^1(ds,dx). 
\end{split}
\end{equation}
We now change variables and set $v^1(t,x)= u^1(t, p(t)-x)$. Equation (\ref{change_var1}) then becomes
\begin{equation*}
\begin{split}
\int_0^1 v^1(t,x) \phi(t,p(t)-x) \textrm{d}x= & \int_0^1 v^1(0,x) \phi(0,p(0)-x) \textrm{d}x +  \int_0^t \int_0^1 v^1(s,x) \frac{\partial \phi}{\partial t}(s,p(s)-x) \textrm{d}x\textrm{d}s \\ &+ \int_0^t \int_0^1 v^1(s,x) \frac{\partial^2 \phi}{\partial x^2}(s,p(s)-x) \textrm{d}x\textrm{d}s
\\ & + \int_0^t \int_0^1 f_1(x,s,v^1(s,\cdot)) \phi(s,p(s)-x) \textrm{d}x\textrm{d}s \\ &+ \int_0^t \int_0^1 \sigma_1(x,s,v^1(s,\cdot)) \phi(s,p(s)-x)W_p(\textrm{d}x,\textrm{d}s)\\ &+  \int_0^t \int_0^1 \phi(s,p(s)-x) \; \eta_p^1(ds,dx). 
\end{split}
\end{equation*}
Here, $\dot{W}_p$ and $\eta_p^1$ are obtained by from $W$ and $\eta$ by defining, for $t \in \mathbb{R}^+$ and $A \in \mathscr{B}( \mathbb{R})$,
\begin{equation*}
\dot{W}_p([0,t] \times A)= \int_0^t \int_{p(s)-A} W(\textrm{d}s,\textrm{d}y), \; \; \; \; \; \; \eta_p^1([0,t] \times A)= \int_0^t \int_{p(s)-A} \eta(\textrm{d}y,\textrm{d}s).
\end{equation*}
Note that, since the process $p(t)$ is $\mathscr{F}_t$-predictable, $\dot{W}_p$ is then also a space time white noise which respects the filtration $\mathscr{F}_t$. Also, $\eta_p^1$ is a reflection measure for $v$, so that
\begin{equation*}
\int_0^T \int_0^1 v^1(t,x) \;  \eta_p^1(\textrm{d}t,\textrm{d}x)=0.
\end{equation*} 
Differentiating $\phi$ in time gives
\begin{equation*}\label{changeofvar}
\frac{\partial \phi}{\partial t}(t,x)= \frac{\partial \varphi}{\partial t}(t,p(t)-x) + p^{\prime}(t)\frac{\partial \varphi}{\partial x}(t,p(t)-x).\end{equation*}
It follows that 
\begin{equation*}
\begin{split}
\int_0^1 v^1(t,x) \varphi(t,x) \textrm{d}x= & \int_0^1 v^1(0,x) \varphi(0,x) \textrm{d}x +  \int_0^t \int_0^1 v^1(s,x) \frac{\partial \varphi}{\partial t}(s,x) \textrm{d}x\textrm{d}s \\ &+ \int_0^t \int_0^1 v^1(s,x)p^{\prime}(s) \frac{\partial \varphi}{\partial x}(s,x) \textrm{d}x\textrm{d}s+ \int_0^t \int_0^1 v^1(s,x) \frac{\partial^2 \varphi}{\partial x^2}(s,x) \textrm{d}x\textrm{d}s
\\ & + \int_0^t \int_0^1 f_1(x,s,v^1(s,\cdot)) \varphi(s,x) \textrm{d}x\textrm{d}s\\ & + \int_0^t \int_0^1 \sigma_1(x,s,v^1(s,\cdot)) \varphi(s,x)W_p(\textrm{d}x,\textrm{d}s) \\ &+ \int_0^t \int_0^1 \varphi(s,x) \; \eta_p^1(ds,dx).
\end{split}
\end{equation*}
We can perform similar manipulations to obtain a weak form for $v^2(t,x):= u^2(t,p(t)+x)$. This gives that for test functions $\varphi \in C_c^{\infty}([0,T] \times (0,1))$, 
\begin{equation*}
\begin{split}
\int_0^1 v^2(t,x) \varphi(t,x) \textrm{d}x= & \int_0^1 v^2(0,x) \varphi(0,x) \textrm{d}x +  \int_0^t \int_0^1 v^2(s,x) \frac{\partial \varphi}{\partial t}(s,x) \textrm{d}x\textrm{d}s \\ & - \int_0^t \int_0^1 v^2(s,x)p'(s) \frac{\partial \varphi}{\partial x}(s,x) \textrm{d}x\textrm{d}s \\ & + \int_0^t \int_0^1 v^2(s,x) \frac{\partial^2 \varphi}{\partial x^2}(s,x) \textrm{d}x\textrm{d}s
+ \int_0^t \int_0^1 f_2(x,s,v^2(s,\cdot)) \varphi(s,x) \textrm{d}x\textrm{d}s \\ & + \int_0^t \int_0^1 \sigma_2(x,s,v^2(s,\cdot)) \varphi(s,x)W_p^-(\textrm{d}x,\textrm{d}s)+ \int_0^t \int_0^1 \varphi(s,x) \;  \eta_p^2(ds,dx),
\end{split}
\end{equation*}
where $\dot{W}_p^-$ is given by
 \begin{equation*}
\dot{W}_p^-([0,t] \times A)= \dot{W}_p([0,t] \times (-A)).
\end{equation*}
We make the observation that, since $(u^1,\eta^1,u^2,\eta^2,p)$ is $\mathscr{F}_t$-adapted, $(v^1,\eta^1_p, v^2, \eta^1_p, p)$ is also $\mathscr{F}_t$-adapted. \\

We now define our notion of solution for a class of reflected SPDEs. This will prove useful when defining solutions to our moving boundary problems, and when proving existence for these via a sequence of solutions to truncated versions of the equations.

\begin{defn}\label{spde solution}

Let $(\Omega, \mathscr{F}, \mathscr{F}_t,\mathbb{P})$ be a complete filtered probability space. Let $\dot{W}$ be a space time white noise on this space which respects the filtration $\mathscr{F}_t$. Suppose that $\tilde{v}$ is a continuous $\mathscr{F}_t$-adapted process taking values in $\mathscr{H}$. Let $F: \mathscr{H}  \rightarrow \mathscr{H}$ and $h: \mathscr{H} \times \mathscr{H} \rightarrow \mathbb{R}$. For the $\mathscr{F}_t$-stopping time $\tau$, we say that the pair $(v, \eta)$ is a local solution to the reflected SPDE

\begin{equation*}
\frac{\partial v}{\partial t}= \Delta v + h(v(t,\cdot), \tilde{v}(t,\cdot)) \frac{\partial F(v)}{\partial x} + f(x,t, v(t,\cdot)) + \sigma(x,t,v(t,\cdot)) \dot{W} + \eta
\end{equation*}
on $[0,\infty) \times [0,1]$ with Dirichlet boundary conditions $v(t,0)=v(t,1)=0$ and initial data $v_0 \in \mathscr{H}^+$, with $v_0 \geq 0$, until time $\tau$, if
\begin{enumerate}[(i)]
\item For every $x \in [0,1]$ and every $t \geq 0$, $v(t,x)$ is an $\mathscr{F}_t$-measurable random variable.
\item $v \geq 0$ almost surely.
\item $v \big|_{[0,\tau) \times [0,1]} \in C([0, \tau) ; \mathscr{H})$ almost surely.
\item $v(t,x)= \infty$ for every $t \geq \tau$ and $x \in [0,1]$ almost surely.
\item $\eta$ is a measure on $(0,1) \times [0,\infty)$ such that
\begin{enumerate}
\item For every measurable map $\psi: [0,1] \times [0,\infty) \rightarrow \mathbb{R}$, 
\begin{equation*}
\int_0^t \int_0^1 \psi(x,s) \; \eta(\textrm{d}x,\textrm{d}s)
\end{equation*}
is $\mathscr{F}_t$-measurable.
\item $\int_0^{\infty} \int_0^1 v(t,x) \; \eta(\textrm{dx,dt})=0$.
\end{enumerate} 
\item There exists a localising sequence of stopping times $\tau_n \uparrow \tau$ almost surely, such that for every $\varphi \in C^{1,2}_c([0,\infty) \times [0,1])$ such that $\varphi(s,0)=\varphi(s,1)=0$ for every $s \in [0,\infty)$, 
\begin{equation}\label{111}
\begin{split}
\int_0^1 v(t \wedge \tau_n,x) \varphi(t,x) \textrm{d}x= & \int_0^1 v(0,x) \varphi(0,x) \textrm{d}x+  \int_0^{t \wedge \tau_n} \int_0^1 v(s,x) \frac{\partial \varphi}{\partial t}(s,x)\textrm{d}x\textrm{d}s \\ & +  \int_0^{t \wedge \tau_n} \int_0^1 v(s,x) \frac{\partial^2 \varphi}{\partial x^2}(s,x) \textrm{d}x\textrm{d}s \\ & - \int_0^{t \wedge \tau_n} \int_0^1 h(v(s,\cdot), \tilde{v}(s,\cdot))F(v(s,x)) \frac{\partial \varphi}{\partial x}(s,x) \textrm{d}x\textrm{d}s\\ & + \int_0^{t \wedge \tau_n} \int_0^1 f(x,s,v(s,\cdot)) \varphi(s,x) \textrm{d}x\textrm{d}s \\ & + \int_0^{t \wedge \tau_n} \int_0^1 \sigma(x,s,v(s,\cdot)) \varphi(s,x)W(\textrm{d}x,\textrm{d}s) \\ & + \int_0^{t \wedge \tau_n} \int_0^1  \varphi(s,x) \; \eta(ds,dx).
\end{split}
\end{equation}
for every $t \geq 0$ almost surely.
\end{enumerate}
\end{defn}
We say that a local solution is \emph{maximal} if there does not exist a solution to the equation on a larger stochastic interval, and we say that a local solution is \emph{global} if we can take $\tau_n = \infty$ in (\ref{111}).

The following definitions provide us with notation which will allow us to easily move between the relative frame (measured with respect to the current position of the boundary) and the fixed frame when discussing solutions to our equations.

\begin{defn}\label{frameshift}
For $p_0 \in \mathbb{R}$, we define $\Theta_{p_0}^1: \mathbb{R} \rightarrow \mathbb{R}$  such that 
$$\Theta_{p_0}^1(x) = p_0 -x.$$
For a function $p: [0,\infty) \rightarrow \mathbb{R}$ we then define $\theta^1_p: [0,\infty) \times \mathbb{R} \rightarrow [0,\infty) \times \mathbb{R}$ such that 
$$\theta^1_p(t,x):= (t,\Theta_{p(t)}^1(x)).$$
We similarly define $\Theta^2_{p_0} : \mathbb{R} \rightarrow \mathbb{R}$ such that $$\Theta_{p_0}^2(x) = x- p_0,$$ and $\theta^2_p : [0,\infty) \times \mathbb{R} \rightarrow [0,\infty) \times \mathbb{R}$ such that 
$$\theta^2_p(t,x):= (t,\Theta_{p(t)}^2(x)).$$
\end{defn}
\begin{defn}
For a space time white noise $\dot{W}$, we denote by $\dot{W}^-$ the space time white noise such that $\dot{W}^-([0,t] \times A) = \dot{W}([0,t] \times (-A))$.
\end{defn}
Before finally stating our definition for solutions to (\ref{movingboundaryoriginal}), we first outline the conditions on our coefficients $f_i$, $\sigma_i$ and $h$.
\begin{enumerate}[(i)]
\item For $i=1,2$, $f_i, \sigma_i: [0,1] \times [0,\infty) \times \mathscr{H} \rightarrow \mathbb{R}$ are measurable mappings.
\item $h: \mathscr{H} \times \mathscr{H} \rightarrow \mathbb{R}$ is a measurable map.
\item For $i=1,2$, $f_i$ is such that for every $T>0$ there exists $C_T$ such that for every $t \in [0,T]$ and every $x \in [0,1]$,
\begin{equation*}
|f_i(x,t,0)| \leq C_T.
\end{equation*}
\item For $i=1,2$, $f_i$ satisfies the local Lipschitz condition that for every $M,T>0$, there exists $C_{T,M}$ such that for every $t \in [0,T]$, every $x \in [0,1]$ and every $u,v \in \mathscr{H}$ with $\|u\|_{\mathscr{H}}, \|v\|_{\mathscr{H}} \leq M$,
\begin{equation*}
|f_i(x,t,u)-f_i(x,t,v)| \leq C_{T,M}\|u-v\|_{\mathscr{H}}.
\end{equation*}
\item For $i=1,2$, $\sigma_i$ is such that for every $T>0$ there exists $C_T$ such that for every $t \in [0,T]$ and every $x \in [0,1]$
\begin{equation*}\label{vol lin}
|\sigma_i(x,t,0)| \leq C_Tx
\end{equation*}
\item For $i=1,2$, $\sigma_i$ satisfies the local Lipschitz condition that for every $M,T>0$, there exists $C_{T,M}$ such that for every $t \in [0,T]$, every $x \in [0,1]$ and every $u,v \in \mathscr{H}$ with $\|u\|_{\mathscr{H}}, \|v\|_{\mathscr{H}} \leq M$,
\begin{equation*}\label{vol lip}
|\sigma_i(x,t,u)-\sigma_i(x,t,v)| \leq C_{T,M}x \|u-v\|_{\mathscr{H}}.
\end{equation*} 
\item $h$ is bounded on bounded sets.
\item $h$ satisfies the local Lipschitz condition that for every $M>0$, there exists $C_M$ such that for every $u,v \in \mathscr{H}$ with $\|u\|_{\mathscr{H}}, \|v\|_{\mathscr{H}} \leq M$,
\begin{equation*}
|h(u_1,v_1)-h(u_2,v_2)| \leq C_M (\|u_1-u_2\|_{\mathscr{H}}+ \|v_1- v_2\|_{\mathscr{H}}).
\end{equation*}
\end{enumerate}
\begin{rem}
An immediate consequence of the conditions on $f_i$ and $\sigma_i$ is that they satisfy local linear growth conditions. That is, for every $T,M>0$ there exists a constant $C_{T,M}$ such that for every $t \in [0,T]$, every $x \in [0,1]$ and every $u \in \mathscr{H}$ with $\|u\|_{\mathscr{H}} \leq M$, 
\begin{equation*}
f_i(x,t,u) \leq C_{T,M}\left( 1+ \|u\|_{\mathscr{H}} \right),
\end{equation*}
\begin{equation*}
\sigma_i(x,t,u) \leq C_{T,M}\left( 1+ \|u\|_{\mathscr{H}} \right).
\end{equation*}
\end{rem}
\begin{rem}
The class of permitted functions for $h$ includes as a subclass functions of the form 
\begin{equation*}
h(u,v)= \tilde{h}\left( \frac{\partial u}{\partial x}(0), \frac{\partial v}{\partial x}(0) \right), 
\end{equation*}
where $\tilde{h}$ is a Lipschitz function from $\mathbb{R} \times \mathbb{R} \rightarrow \mathbb{R}$.
\end{rem}

\begin{rem}
The $x$ term in the linear growth and Lipschitz conditions for $\sigma$ ensure that the volatility will decay at least linearly as we approach the boundary i.e. as $x \downarrow 0$. We note that, in the case when $\sigma(x,t,u)= \tilde{\sigma}(x,t,u(x))$, the typical Lipschitz condition 
\begin{equation*}
|\tilde{\sigma}(x,t,u) - \tilde{\sigma}(x,t,v)| \leq C|u-v|
\end{equation*}
implies the Lipschitz condition (\ref{vol lip}). In particular, volatility functions of the form 
\begin{equation*}
\sigma(x,t,u)= \sigma_1(x) + \sigma_2(x)u(x),
\end{equation*}
where $\frac{\sigma_1}{x}, \sigma_2 \in L^{\infty}([0,1])$, are permitted.
\end{rem}

We are now in position to state our definition for a solution to problem (\ref{movingboundaryoriginal}). Following this, we state the main result, namely existence and uniqueness for maximal solutions to these equations.

\begin{defn}
Let $(\Omega, \mathscr{F}, \mathscr{F}_t,\mathbb{P})$ be a complete filtered probability space, and $\dot{W}$ be a space time white noise on this space which respects the filtration $\mathscr{F}_t$. Suppose that $p_0 \in \mathbb{R}$, and for $i=1,2$, $u_0^i$ is such that $u_0^i \circ (\Theta_{p_0}^i)^{-1} \in \mathscr{H}^+$. We say that the quintuple $(u^1, \eta^1, u^2, \eta^2, p)$ satisfies the reflected stochastic Stefan problem
\begin{equation*}
\begin{split}
& \frac{\partial u^1}{\partial t}= \Delta u^1 + f_1(p(t)-x,t, u^1(t,p(t)-\cdot))+ \sigma_1(p(t)-x,t,u^1(t,p(t)-\cdot))\dot{W} + \eta^1 \\ 
& \frac{\partial u^2}{\partial t}= \Delta u^2 + f_2(x-p(t),t,u^2(t, p(t)+ \cdot))+ \sigma_2(x-p(t),t,u^2(t, p(t)+ \cdot))\dot{W} + \eta^2,
\\ & p^{\prime}(t) = h(u^1(t,p(t)-\cdot), u^2(t,p(t)+\cdot)),
\end{split}
\end{equation*}
with initial data $(u^1_0,u^2_0, p_0)$, up to the $\mathscr{F}_t$-stopping time $\tau$ if
\begin{enumerate}[(i)]
\item $(v^1, \tilde{\eta}^1):=(u^1 \circ (\theta^1_p)^{-1}, \eta^1 \circ (\theta^1_p)^{-1})$ solves, in the sense of Definition \ref{spde solution}, the reflected SPDE
\begin{equation*}
\frac{\partial v^1}{\partial t}= \Delta v^1 - h(v^1(s,\cdot),v^2(s,\cdot) )\frac{\partial v^1}{\partial x} + f_1(x,t,v^1(t,\cdot)) + \sigma_1(x,t,v^1(t,\cdot))\dot{W} + \tilde{\eta}^1 
\end{equation*}
with Dirichlet boundary conditions $v^1(0)=v^1(1)=0$ and initial data $u_0^1 \circ (\Theta_{p_0}^1)^{-1}$, until time $\tau$.
\item $(v^2, \tilde{\eta}^2):=(u^2 \circ (\theta^2_p)^{-1}, \eta^2 \circ (\theta^2_p)^{-1})$ solves, in the sense of Definition \ref{spde solution}, the reflected SPDE
\begin{equation*}
\frac{\partial v^2}{\partial t}= \Delta v^2 + h(v^1(s,\cdot),v^2(s,\cdot) )\frac{\partial v^2}{\partial x} + f_2(x,t,v^2(t,\cdot)) + \sigma_2(x,t,v^2(t,\cdot))\dot{W}^- + \tilde{\eta}^2 
\end{equation*}
with Dirichlet boundary conditions $v^2(0)=v^2(1)=0$ and initial data $u_0^2 \circ (\Theta_{p_0}^2)^{-1}$, until time $\tau$.
\item $p(t)=p_0 + \int_0^t h(v^1(s,\cdot),v^2(s,\cdot) ) \; \textrm{d}s$.
\end{enumerate}
We refer to $(v^1, \tilde{\eta}^1, v^2, \tilde{\eta}^2)$ as the solution to the moving boundary problem in the relative frame.
\end{defn}
\begin{thm}\label{main}
There exists a unique maximal solution to the reflected stochastic Stefan problem.
\end{thm}
\begin{rem}
It is immediate from the definition that existence of a unique maximal solution to the reflected stochastic Stefan problem is equivalent to the existence of a unique maximal solution to the system of coupled reflected SPDEs
\begin{equation}\label{movingboundary}\begin{split}
& \frac{\partial v^1}{\partial t} = \Delta v^1 -h(v^1(t, \cdot), v^2(t,\cdot) )\frac{\partial v^1}{\partial x}+f_1(x, t, v^1(t,\cdot) )+ \sigma_1(x,t, v^1(t,\cdot))\dot{W} + \eta^1, \\ 
& \frac{\partial v^2}{\partial t} = \Delta v^2 +h(v^1(t, \cdot), v^2(t,\cdot) )\frac{\partial v^2}{\partial x}+ f_2(x,t, v^2(t,\cdot) )+ \sigma_2(x,t, v^2(t,\cdot) )\dot{W}^- + \eta^2.
\end{split}
\end{equation}
\end{rem}

\section{The Deterministic Obstacle Problem and the Corresponding Bounds in $\mathscr{H}$}
We begin this section by discussing some relevant work on a deterministic obstacle problem.  The obstacle problem in the form given here was originally discussed in \cite{NP}. The space $C_0((0,1))$ here denotes the space of continuous functions on $(0,1)$ which vanish at the endpoints.

\begin{defn}
Let $z \in C([0,T] \times [0,1])$ with $z(t,\cdot) \in C_0((0,1))$ for every $t \in [0,T]$ and $z(0, \cdot) \leq 0$. We say that the pair $(w, \eta)$ satisfies the heat equation with obstacle $z$ if:

\begin{enumerate}[(i)]
\item $w \in C([0,T] \times [0,1])$, $w(t,0)=w(t,1)=0$, $w(0,x) \equiv 0$ and $w \geq z$.
\item $\eta$ is a measure on $(0,1) \times [0,T]$.
\item $w$ weakly solves the PDE 
\begin{equation*}
\frac{\partial w}{\partial t}= \frac{\partial^2 w}{\partial x^2} + \eta
\end{equation*}
That is, for every $t \in [0,T]$ and every $\phi \in C^2((0,1)) \cap C_0((0,1))$, $$\int_0^1 w(t,x) \phi(x) \textrm{d}x= \int_0^t \int_0^1 w(s,x)\phi^{\prime \prime}(x) \textrm{d}x\textrm{d}s + \int_0^t \int_0^1 \phi(x) \; \eta(\textrm{d}x,\textrm{d}s).$$

\item $\int_0^t \int_0^1 (w(s,x)-z(s,x)) \; \eta(\textrm{d}x,\textrm{d}s)=0$.

\end{enumerate}
\end{defn}
Existence and uniqueness for solutions to this problem was proved in \cite{NP}. It was also shown that the difference between two solutions can be bounded in the $L^{\infty}$-norm by the difference in the $L^{\infty}$-norm of the obstacles. We now adapt this result, and show that solutions to the obstacle problem in the case when the obstacle lies in $C([0,T];\mathscr{H})$ are themselves in the space $C([0,T];\mathscr{H})$. We also prove the corresponding estimate, that we can control the $C([0,T];\mathscr{H})$-norm of the difference in the solutions by the $C([0,T];\mathscr{H})$-norm of the difference of the obstacle. This will be required to prove existence for our reflected SPDEs via a Picard argument.

\begin{thm}\label{Obstacle H}
Suppose that $z_1,z_2 \in C([0,T] ; \mathscr{H})$. Let $w_1, w_2$ solve the corresponding parabolic obstacle problems with obstacles $z_1$ and $z_2$. Then we have that $w_1, w_2 \in C([0,T]; \mathscr{H})$ and 
\begin{equation*}
\|w_1 - w_2\|_{\mathscr{H},T} \leq \|z_1 - z_2\|_{\mathscr{H},T}.
\end{equation*}
\end{thm}
\begin{proof}
Define $\delta := \|z_1-z_2\|_{\mathscr{H},T}$. For $i=1,2$ and $\epsilon >0$, let $w_i^{\epsilon}$ solve the following penalised equations
\begin{equation*}\label{PDE}
\frac{\partial w_i^{\epsilon}}{\partial t} = \Delta w_i^{\epsilon} + g_{\epsilon}(w_i^{\epsilon} - z),
\end{equation*}
with Dirichlet boundary conditions $w_i^{\epsilon}(t,0) = w_i^{\epsilon}(t,1) = 0$ and zero initial data, where we define $$g_{\epsilon}(x):= \frac{1}{\epsilon} \arctan(( x \wedge 0)^2).$$
Then we know that $w_i^{\epsilon} \uparrow w_i$ for $i=1,2$ (see, for example, the proof of Theorem 1.4 in \cite{NP}). We define $k^{\epsilon}(t,x):= w_1^{\epsilon}- w_2^{\epsilon} - x \delta.$ Then $k^{\epsilon}$ solves the equation
\begin{equation*}
\frac{\partial k^{\epsilon}}{\partial t} = \Delta k^{\epsilon} + g_{\epsilon}(w_1^{\epsilon} -z_1)- g_{\epsilon}(w_2^{\epsilon} -z_2),
\end{equation*}
with boundary conditions $k^{\epsilon}(t,0)=0$, $k^{\epsilon}(t,1)=-\delta$ and negative initial data $k^{\epsilon}(0,x)= -x \delta$. Testing this equation with $(k^{\epsilon})^+$ and integrating over space and time, we obtain that 
\begin{equation}\label{ob_ep}
\begin{split}
\frac{1}{2}\| (k^{\epsilon}_T)^+ \|_{L^2}^2 = &  \frac{1}{2}\| (k^{\epsilon}_0)^+ \|_{L^2}^2 -\int_0^T \int_0^1 \left| \frac{\partial (k_t^{\epsilon})^+}{\partial x}\right|^2 \textrm{d}x\textrm{d}t + \int_0^T \int_0^1 g_{\epsilon}(w_1^{\epsilon}(t,x)-z_1(t,x))(k^{\epsilon}(t,x))^+ \textrm{d}x \textrm{d}t \\ & - \int_0^T \int_0^1 g_{\epsilon}(w_2^{\epsilon}(t,x)-z_2(t,x))(k^{\epsilon}(t,x))^+ \textrm{d}x \textrm{d}t.
\end{split}
\end{equation}
Note that $(k_0^{\epsilon})^+ =0$ and $g_{\epsilon}$ is a decreasing function. Also, on the set $k^{\epsilon} \geq 0$, $w_1^{\epsilon} - z_1 \geq w^{\epsilon}_2 - z_2$. It follows that the right hand side of (\ref{ob_ep}) is negative. Therefore, $(k^{\epsilon})^+=0$, and so 
\begin{equation*}
w_1^{\epsilon}-w_2^{\epsilon} \leq x \delta = x \|z_1 - z_2 \|_{\mathscr{H},T}.
\end{equation*}
Interchanging $w_1^{\epsilon}$ and $w_2^{\epsilon}$, we obtain that, for every $\epsilon >0$,
\begin{equation*}
\sup\limits_{t \in [0,T]} \sup\limits_{x \in (0,1]} \left| \frac{w_1^{\epsilon}(t,x)}{x}- \frac{w_2^{\epsilon}(t,x)}{x} \right| \leq \|z_1 - z_2 \|_{\mathscr{H},T}.
\end{equation*}
Letting $\epsilon \downarrow 0$, we deduce that
\begin{equation}\label{20}
\sup\limits_{t \in [0,T]} \sup\limits_{x \in (0,1]} \left| \frac{w_1(t,x)}{x}- \frac{w_2(t,x)}{x} \right| \leq \|z_1 - z_2 \|_{\mathscr{H},T}.
\end{equation}
We now argue that the functions $w_i \in C([0,T] ; \mathscr{H})$. Note that we have $C([0,T];\mathscr{H})$ regularity for solutions to the obstacle problem provided that the obstacle is smooth (see, for example, Theorem 4.1 in \cite{AE}). Let $z_n$ be a sequence of smooth obstacles such that $z_n \rightarrow z$ in $C([0,T];\mathscr{H})$, and denote by $w_n$ the solution to the obstacle problem with obstacle $z_n$. Then, by (\ref{20}), the sequence $w_n$ is Cauchy in $C([0,T];\mathscr{H})$, and so converges in $C([0,T];\mathscr{H})$ to some $\tilde{w}$. We also have that 
\begin{equation*}
\sup\limits_{t \in [0,T]} \sup\limits_{x \in (0,1]} \left| \frac{w_n(t,x)}{x}- \frac{w(t,x)}{x} \right| \leq \|z_n -z\|_{\mathscr{H},T} \rightarrow 0.
\end{equation*}
Therefore, we have that $\tilde{w}=w$, and so $w \in C([0,T];\mathscr{H})$. This concludes the proof.
\end{proof}
\section{Key Estimates}

In order to prove existence for our reflected SPDE, we will  perform a Picard iteration in the space $L^p(\Omega; C([0,T]; \mathscr{H}))$ to show that solutions of (\ref{movingboundary}) exist. Having defined $(v^{1,n},\eta^{1,n},v^{2,n},\eta^{2,n})$, we will define our $(n+1)^{\textrm{th}}$ approximation of a solution to be given by the solutions of 
\begin{equation*}\begin{split}
\frac{\partial v^{1,n+1}}{\partial t} = & \Delta v^{1,n+1} + f(x,t,v^{1,n}(t,\cdot)) - \frac{\partial v^{1,n}}{\partial x}p_{n+1}^{\prime}(t) + \sigma(x,t,v^{1,n}(t,\cdot)) \dot{W}+ \eta^{1,n+1},
\end{split}
\end{equation*}
\begin{equation*}\begin{split}
\frac{\partial v^{2,n+1}}{\partial t} = & \Delta v^{2,n+1} + f(x,t,v^{2,n}(t,\cdot)) + \frac{\partial v^{2,n}}{\partial x}p_{n+1}^{\prime}(t) + \sigma(x,t,v^{2,n}(t,\cdot)) \dot{W}^- + \eta^{2,n+1}.
\end{split}
\end{equation*}
\begin{equation*}
p_{n+1}(t) = p(0) + \int_0^t h\left( v^{1,n}(s,\cdot), v^{2,n}(s,\cdot)  \right) \textrm{d}s.
\end{equation*}
In order to control the difference $\mathbb{E}\left[ \|v^{1,n+1}-v^{1,n}\|^p_{\mathscr{H},T} \right]$, we will use Theorem \ref{Obstacle H}. This gives
\begin{equation*}
\mathbb{E}\left[ \|v^{1,n+1}-v^{1,n}\|^p_{\mathscr{H},T} \right] \leq 2^p \mathbb{E}\left[ \|z^{1,n+1}-z^{1,n}\|^p_{\mathscr{H},T} \right],
\end{equation*}
where $z^{1,n+1}$ solves
\begin{equation*}
\frac{\partial z^{1,n+1}}{\partial t} = \Delta z^{1,n+1} - \frac{\partial v^{1,n}}{\partial x}p_{n+1}^{\prime}(t)+ f(x,t,v^{1,n}(t,\cdot)) + \sigma(x,t, v^{1,n}(t,\cdot)) \frac{\partial^2 W}{\partial x \partial t}.
\end{equation*}
We deduce in the same way that $\mathbb{E}\left[ \|v^{2,n+1}-v^{2,n}\|^p_{\mathscr{H},T} \right] \leq C_p\mathbb{E}\left[ \|z^{2,n+1}-z^{2,n}\|^p_{\mathscr{H},T} \right],$ where $z^{2,n+1}$ solves
\begin{equation*}
\frac{\partial z^{2,n+1}}{\partial t} = \Delta z^{2,n+1} + \frac{\partial v^{2,n}}{\partial x}p_{n+1}^{\prime}(t)+ f(x,t, v^{2,n}(t,\cdot)) + \sigma(x,t, v^{2,n}(t,\cdot)) \frac{\partial^2 W}{\partial x \partial t}.
\end{equation*}
By writing this in mild form, noting that the mild and weak forms are equivalent for such equations, we see that we will require estimates on the terms 
\begin{enumerate}[(i)]
\item \begin{equation}\label{A1}
\int_0^t \int_0^1 \frac{1}{x}G(t-s,x,y)\left[ f(x, s, v^{i,n}(s,\cdot))-f(x,s, v^{i-1,n}(s,\cdot))\right] \textrm{d}y \textrm{d}s.
\end{equation}
\item \begin{equation}\label{A2}
\int_0^t \int_0^1 \frac{y}{x}G(t-s,x,y)\left[ \frac{\sigma(x, s, v^{i,n}(s,\cdot))-\sigma(x, s, v^{i-1,n}(s,\cdot))}{y}\right] W(\textrm{d}y, \textrm{d}s).
\end{equation}
\item \begin{equation}\label{A3}
\int_0^t \int_0^1 \frac{y}{x}\frac{\partial G}{\partial y}(t-s,x,y)\left[\frac{v^{i,n}(s,y)}{y}(p^n)^{\prime}(s)-\frac{v^{i,n-1}(s,y)}{y}(p^{n-1})^{\prime}(s)\right] \textrm{d}y \textrm{d}s.
\end{equation}
\end{enumerate}
Obtaining such estimates will be the focus in this section.

\subsection{Heat Kernel Estimates}
In this section we will state and prove the heat kernel estimates which will be crucial in proving existence and uniqueness for our equations. In particular, they allow us to show that solutions to certain SPDEs lie in the space $C([0,T];\mathscr{H})$, and also enable us to obtain estimates in $L^p(\Omega; C([0,T]; \mathscr{H}))$ for these solutions. We define $G$ to be the Dirichlet heat kernel on $[0,1]$, so 
\begin{equation*}
G(t,x,y):= \frac{1}{\sqrt{4 \pi t}}\sum\limits_{n =- \infty}^{\infty} \left[ \exp\left( - \frac{(x-y+2n)^2}{4t} \right) - \exp\left(- \frac{(x+y+2n)^2}{4t} \right) \right].
\end{equation*}
We then define for $t >0$, $y \in [0,1]$ and $x \in (0,1]$
\begin{equation*}
\tilde{G}(t,x,y):= \frac{y}{x}G(t,x,y).
\end{equation*}
For $t>0$, $y \in [0,1]$ and $x=0$, we set
\begin{equation*}
\tilde{G}(t,x,y):= y\frac{\partial G}{\partial x}(t,0,y).
\end{equation*}
We also define 
\begin{equation*}
\tilde{H}(t,x,y):= \frac{y}{x}\frac{\partial G}{\partial y}(t,x,y)
\end{equation*}
for $t \in [0,T]$, $x \in (0,1]$ and $y \in [0,1]$, and set 
\begin{equation*}
\tilde{H}(t,0,y):= y \frac{\partial^2 G}{\partial y \partial x}(t,0,y)
\end{equation*}
for $t \in [0,T]$ and $y \in [0,1]$.
\begin{rem}\label{G}
Note that 
\begin{equation*}\begin{split}
G(t,x,y)= & \frac{1}{\sqrt{4 \pi t}} \left[ \exp\left( - \frac{(x-y)^2}{4t} \right) - \exp\left(- \frac{(x+y)^2}{4t} \right) \right]\\ & +\frac{1}{\sqrt{4 \pi t}} \left[ \exp\left( - \frac{(x-y-2)^2}{4t} \right) - \exp\left(- \frac{(x+y-2)^2}{4t} \right) \right] + L(t,x,y),
\end{split}
\end{equation*}
where $L$ is a smooth function on $[0,T] \times [0,1] \times [0,1]$, vanishing on $t=0$ and $x=0,1$. Since the arguments for bounding the second term on the right hand side are similar to those for bounding the first, and estimates for $L$ are simple due to $L$ being smooth, it is sufficient to prove the estimates for the first term,
\begin{equation*}
G_1(t,x,y):=\frac{1}{\sqrt{4 \pi t}} \left[ \exp\left( - \frac{(x-y)^2}{4t} \right) - \exp\left(- \frac{(x+y)^2}{4t} \right) \right].
\end{equation*}
\end{rem}

\begin{prop}\label{Heat Kernel}
The following estimates hold for $G$, $\tilde{G}$ and $\tilde{H}$.
\begin{enumerate}
\item For $T>0$, $\exists \; C_T >0$ such that for every $t \in [0,T]$
\begin{equation*}
\sup\limits_{x \in (0,1]}\int_0^1 \left| \frac{1}{x}G(t,x,y) \right| \textrm{d}y \leq \frac{C_T}{\sqrt{t}}.
\end{equation*}
\item For $T>0$, $\exists \; C_T >0$ such that for every $t \in [0,T]$,  $$\sup\limits_{x \in [0,1]} \int_0^1 \tilde{G}(s,x,y)^2 \textrm{d}y \leq \frac{C_T}{\sqrt{t}}.$$
\item For $T >0$ and $q \in (1,2)$, $\exists \; C_{T,q}>0$ such that, for every $x, y \in [0,1]$
\begin{equation*}
\sup\limits_{t \in [0,T]} \int_0^t \left[ \int_0^1 \left( \tilde{G}(s,x,z)- \tilde{G}(t,y,z) \right)^2 \textrm{d}y \right]^q \textrm{d}s \leq C_{T,q}|x-y|^{(2-q)/3}.
\end{equation*}
\item For $T>0$ and $q \in (1,2)$, $\exists \; C_{T,q} >0$ such that for $0 \leq s \leq t \leq T$ 
\begin{equation*}
\sup\limits_{x \in [0,1]} \int_0^s \left[ \int_0^1 (\tilde{G}(t-r,x,y)- \tilde{G}(s-r,x,y))^2 \textrm{d}y \right]^q \textrm{d}r \leq C_{T,q}|t-s|^{(2-q)/2}.
\end{equation*}
\item For $T>0$, $\exists C_T>0$ such that for every $t \in [0,T]$
 \begin{equation*}
\sup\limits_{x \in [0,1]} \int_0^1 \left| \tilde{H}(t,x,y) \right| \textrm{d}y \leq \frac{C_T}{\sqrt{t}}.
\end{equation*}
\end{enumerate}
\end{prop}
\begin{proof}
The proof of inequality (2) can be found in the proof of Lemma 3.2.1 in \cite{Zheng}. The other inequalities can be shown by adapting some of the arguments from \cite{Zheng}, and their proofs are deferred to the appendix.
\end{proof}
\subsection{Mild form Estimates and Continuity}
We would now like to translate the estimates on the heat kernel into estimates on the terms (\ref{A1}), (\ref{A2}) and (\ref{A3}) appearing in the mild form for the problem. We would also like to ensure continuity of these terms. In the case of the drift and moving boundary terms in the mild form, this is essentially an immediate consequence of the heat kernel estimates. For the stochastic term, more work is required, as we need to control the $L^p$-norm of the supremum of the term over the space-time interval $[0,T] \times [0,1]$. We can achieve this by a standard application of the Garsia-Rodemich-Rumsey Lemma. 

\begin{prop}\label{drift1}
Let $p>2$ and suppose that $f \in L^p(\Omega; L^{\infty}([0,T] \times [0,1]))$. Define $F(t,x)$ such that, for $t \in [0,T]$ and $x \in [0,1]$, 
\begin{equation*}
F(t,x):=\int_0^t \int_0^1 G(t-s,x,y)f(s,y) \textrm{d}y \textrm{d}y.
\end{equation*}
Then we have that  
\begin{enumerate}
\item $F \in C([0,T];\mathscr{H})$ almost surely.
\item For every $t \in [0,T]$, 
\begin{equation*}
\mathbb{E}\left[ \|F\|_{\mathscr{H},t}^p \right] \leq C_{p,T} \int_0^t \mathbb{E}\left[\|f\|_{\infty,s}^p \right] \textrm{d}s.
\end{equation*}
\end{enumerate}
\end{prop}
\begin{proof}
$F \in C([0,T]; \mathscr{H})$ is clear, since $F$ is simply the solution to the heat equation with Dirichlet boundary conditions at $x=0,1$, initial data $0$ and source term $f$. For the second part, we note that for $0 \leq s \leq t \leq T$ and $x \in (0,1]$,
\begin{equation}\label{eqn_new}\begin{split}
\left| \int_0^s \int_0^1 \frac{1}{x}G(s-r,x,y)f(r,y) \; \textrm{d}y \textrm{d}r \right|^p \leq & \left| \int_0^s \left( \int_0^1 \frac{1}{x}G(s-r,x,y)\textrm{d}y \right) \|f\|_{\infty,r} \; \textrm{d}r \right|^p 
\\ \leq &  \left[\int_0^s \left(\int_0^1 \frac{1}{x}G(s-r,x,y)\textrm{d}y \right)^q \textrm{d}r \right]^{p/q} \times \int_0^s \|f\|_{\infty,r}^p \; \textrm{d}r,
\end{split}
\end{equation}
where $q=p/(p-1)$. By estimate (1) in Proposition \ref{Heat Kernel}, we have that 
\begin{equation*}
\left[\int_0^s \left(\int_0^1 \frac{1}{x}G(s-r,x,y)\textrm{d}y \right)^q \textrm{d}r \right]^{p/q} \leq C_{T,p} \left[ \int_0^s \left(\frac{1}{\sqrt{s-r}} \right)^q \textrm{d}r \right]^{p/q} \leq C_{T,p}.
\end{equation*}
Taking the supremum over $s \in [0,t]$ and $x \in (0,1]$ in the inequality (\ref{eqn_new}) then gives that 
\begin{equation*}
\|F \|_{\mathscr{H},t}^p \leq C_{T,p}\int_0^t \|f \|_{\infty, r}^p \; \textrm{d}r.
\end{equation*}
We can then take expectations of both sides to conclude the argument.
\end{proof}
\begin{prop}\label{drift2}
Let $p>2$ and suppose that $f \in L^p(\Omega; C([0,T];\mathscr{H}))$. Define $J(t,x)$ so that 
\begin{equation*}
J(t,x):= \int_0^t \int_0^1 \frac{\partial G}{\partial y}(t-s,x,y)f(s,y) \textrm{d}y \textrm{d}s.
\end{equation*}
Then we have that
\begin{enumerate}
\item $J \in C([0,T];\mathscr{H})$ almost surely.
\item For  $t \in [0,T]$,
\begin{equation*}
\mathbb{E}\left[ \|J\|_{\mathscr{H},t}^p \right] \leq C_{T,p} \int_0^t \mathbb{E}\left[ \|f\|_{\mathscr{H},s}^p \right] \textrm{d}s.
\end{equation*}
\end{enumerate}
\end{prop}
\begin{proof}
We have that, for $x \in (0,1]$ and $t \in [0,T]$,
\begin{equation*}\begin{split}
\frac{J(t,x)}{x}:= &  \int_0^t \int_0^1 \frac{y}{x}\frac{\partial G}{\partial y}(t-s,x,y)\frac{f(s,y)}{y} \textrm{d}y \textrm{d}s 
\\ = & \int_0^t \int_0^1 \frac{y}{x}\frac{\partial G}{\partial y}(s,x,y)\frac{f(t-s,y)}{y} \textrm{d}y \textrm{d}s. 
\end{split}
\end{equation*}
Note that, by the estimates in Proposition A.6 in \cite{Paper}, we know that $\frac{J(t,x)}{x}$ is well defined by our integral expression and continuous on $[0,T] \times (0,1]$. So it is left to show that $\frac{J(t,x)}{x}$ can be extended to a continuous function on $[0,T] \times [0,1]$. In the following arguments, we focus on the $G_1$ component of $G$ (see Remark \ref{G}). Suppose that $(t_n,x_n)$ is a sequence in $[0,T] \times (0,1]$ converging to $(t,0)$. Note that
\begin{equation*}\begin{split}
\int_0^{t_n} \int_0^1 & \mathbbm{1}_{\left\{y < 2x_n \right\}} \frac{y}{x_n} \left| \frac{\partial G_1}{\partial y}(s,x_n,y)\frac{f(t_n-s,y)}{y} \right| \textrm{d}y \textrm{d}s 
\\ \leq & \|f\|_{\mathscr{H},T}\int_0^{t_n} \int_0^1 \mathbbm{1}_{\left\{y < 2x_n \right\}} \frac{y}{x_n} \left| \frac{\partial G_1}{\partial y}(s,x_n,y) \right| \textrm{d}y \textrm{d}s  
\\ \leq &  2\|f\|_{\mathscr{H},T}\int_0^{t_n} \int_0^1 \mathbbm{1}_{\left\{y < 2x_n \right\}} \left| \frac{\partial G_1}{\partial y}(s,x_n,y) \right| \textrm{d}y \textrm{d}s.
\end{split}
\end{equation*}
For $p >3$, we have that this is at most
\begin{equation}\label{bound1}
2\|f\|_{\mathscr{H},T} \left(\int_0^{t_n} \int_0^1 \mathbbm{1}_{\left\{y < 2x_n \right\}} \textrm{d}y \textrm{d}s \right)^{1/p} \times \left( \int_0^{t_n} \int_0^1 \left|\frac{\partial G_1}{\partial y}(s,x_n,y) \right|^q \textrm{d}y \textrm{d}s \right)^{1/q},
\end{equation}
where $q = p/(p-1)$. Note that $q \in (1,3/2)$, and so
\begin{equation*}
\sup\limits_{t \in [0,T]} \sup\limits_{x \in [0,1]}\int_0^{t} \int_0^1 \left|\frac{\partial G_1}{\partial y}(s,x,y) \right|^q \textrm{d}y \textrm{d}s < \infty.
\end{equation*}
We therefore have that (\ref{bound1}) converges to zero as $n \rightarrow \infty$.
Now note that 
\begin{equation*}
\mathbbm{1}_{\left\{y \geq 2x_n, s \in [0,t_n]\right\}}\frac{y}{x_n} \frac{\partial G_1}{\partial y}(s,x_n,y)\frac{f(t_n-s,y)}{y} \rightarrow \mathbbm{1}_{\left\{s \in [0,t]\right\}} y\frac{\partial^2 G_1}{\partial y \partial x}(s,0,y)\frac{f(t-s,y)}{y}
\end{equation*}
almost everywhere in $[0,T] \times [0,1]$. In addition we have that for $s \in [0,T]$, $y \in [0,1]$ and $n \geq 1$,
\begin{equation}\label{calc1}\begin{split}
& \left| \frac{y}{x_n} \frac{\partial G_1}{\partial y}(s,x_n,y) \right| 
\\ & =  C\left| \frac{y}{x_ns \sqrt{s}} \left( (x_n-y)e^{-(x_n-y)^2/4s}+ (x_n+y)e^{-(x_n+y)^2/4s} \right) \right| 
\\ & \leq C\left|\frac{y}{s \sqrt{s}} \left(e^{-(x_n-y)^2/4s}+e^{-(x_n+y)^2/4s} \right) \right| + C\left|\frac{y^2}{x_n s \sqrt{s}} \left( e^{-(x_n-y)^2/4s} - e^{-(x_n+y)/4s} \right)  \right| 
\\ &  \leq C \left|\frac{y}{s \sqrt{s}} \left(e^{-(x_n-y)^2/4s}+e^{-(x_n+y)^2/4s} \right) \right| + C\left|\frac{y^2}{x_n s \sqrt{s}} e^{-(x_n-y)^2/4s}\left( 1 - e^{-x_ny/s} \right)  \right|.
\end{split}
\end{equation}
On the set $\{y \geq 2x_n \}$, we have that 
\begin{equation*}
e^{-(x_n-y)^2/4s} \leq e^{- y^2/16s}.
\end{equation*} 
Note also that, for $x \geq 0$, $1-e^{-x} \leq x$. Therefore, we can bound (\ref{calc1}) on the set $\{y \geq 2x_n\}$ by 
\begin{equation*}\begin{split}
\left|\frac{y}{s \sqrt{s}} \left( e^{-y^2/16s} + e^{-y^2/4s} \right) \right| + \left| \frac{y^2}{x_ns \sqrt{s}} e^{-y^2/16s} \times \frac{x_ny}{s} \right| \leq & C \left( \frac{y}{s \sqrt{s}} e^{-y^2/16s} + \frac{y^3}{s^2 \sqrt{s}}e^{-y^2/16s} \right).
\end{split}
\end{equation*}
We can similarly bound the other components of $G$, to obtain that, for every $n$, on the set $\{y \geq 2x_n \}$,
\begin{equation}\label{dom1}
 \left| \frac{y}{x_n} \frac{\partial G}{\partial y}(s,x_n,y) \right| \leq C \left( \frac{y}{s \sqrt{s}} e^{-y^2/16s} + \frac{y^3}{s^2 \sqrt{s}}e^{-y^2/16s} \right).
\end{equation}
It follows that, for every $n$,
\begin{equation*}
\left| \mathbbm{1}_{\left\{y \geq 2x_n, s \in [0,t_n]\right\}}\frac{y}{x_n} \frac{\partial G}{\partial y}(s,x_n,y)\frac{f(t_n-s,y)}{y} \right| \leq C \left( \frac{y}{s \sqrt{s}} e^{-y^2/16s} + \frac{y^3}{s^2 \sqrt{s}}e^{-y^2/16s} \right) \|f\|_{\mathscr{H},T}
\end{equation*} 
This function is integrable, and so we can apply the DCT to obtain that 
\begin{equation*}\begin{split}
\int_0^{t_n} & \int_0^1  \mathbbm{1}_{\left\{y \geq 2x_n \right\}} \frac{y}{x_n}  \frac{\partial G}{\partial y}(t_n-s,x_n,y)\frac{f(s,y)}{y} \textrm{d}y \textrm{d}s \\ = &  \int_0^{t_n} \int_0^1 \mathbbm{1}_{\left\{y \geq 2x_n \right\}} \frac{y}{x_n}  \frac{\partial G}{\partial y}(s,x_n,y)\frac{f(t_n-s,y)}{y} \textrm{d}y \textrm{d}s \rightarrow \int_0^t \int_0^1 \frac{\partial^2 G}{\partial y \partial x}(s,0,y)f(t-s,y) \textrm{d}y \textrm{d}s 
\\ = & \int_0^t \int_0^1 \frac{\partial^2 G}{\partial y \partial x}(t-s,0,y)f(s,y) \textrm{d}y \textrm{d}s.
\end{split}
\end{equation*}
Another application of the DCT (note that we can use the same dominating function as in (\ref{dom1})) gives that, if $t_n$ is a sequence in $[0,T]$ and $t_n \rightarrow t$, then 
\begin{equation*}
\int_0^{t_n} \int_0^1 \frac{\partial^2 G}{\partial y \partial x}(t_n-s,0,y)f(s,y) \textrm{d}y \textrm{d}s \rightarrow \int_0^t \int_0^1 \frac{\partial^2 G}{\partial y \partial x}(t-s,0,y)f(s,y) \textrm{d}y \textrm{d}s.
\end{equation*}
So we have shown that $J \in C([0,T]; \mathscr{H})$ almost surely. We now prove the bound (2). Let $0 \leq s \leq t \leq T$. For $p>2$, we have that 
\begin{equation*}
\left| \frac{J(s,x)}{x} \right| \leq \left(\int_0^s \left[ \int_0^1 \left| \frac{y}{x} \frac{\partial G}{\partial y}(r,x,y) \right| \textrm{d}y \right]^q \textrm{d}r \right)^{1/q} \times \left( \int_0^s \|f \|_{\mathscr{H},r}^p \textrm{d}r \right)^{1/p}.
\end{equation*}
Arguing as before, considering the cases $y<2x$ and $y \geq 2x$ separately, we have that
\begin{equation*}\label{12}
\left| \frac{y}{x} \frac{\partial G}{\partial y}(r,x,y) \right| \leq C \left( \left| \frac{\partial G}{\partial y}(r,x,y) \right| +  \frac{y}{r \sqrt{r}} e^{-y^2/16r} + \frac{y^3}{r^2 \sqrt{r}}e^{-y^2/16r} \right).
\end{equation*}
For $p>2$ and corresponding $q=p/(p-1) \in (1,2)$, and we have that
\begin{equation*}\label{11}
\begin{split}
\int_0^s &  \left[ \int_0^1 \left( \left| \frac{\partial G}{\partial y}(r,x,y) \right| +  \frac{y}{r \sqrt{r}} e^{-y^2/16r} + \frac{y^3}{r^2 \sqrt{r}}e^{-y^2/16r} \right) \textrm{d}y \right]^q \textrm{d}r 
\\ & \leq C \int_0^s \left(\frac{1}{\sqrt{r}}\right)^q \textrm{d}r \leq C_{T,p}< \infty.
\end{split}
\end{equation*}
It follows that, for $p>2$, 
\begin{equation*}
\left| \frac{J(s,x)}{x} \right| \leq C_{T,p} \left( \int_0^s \|f \|_{\mathscr{H},r}^p \textrm{d}r \right)^{1/p}.
\end{equation*}
Therefore, 
\begin{equation*}
\sup\limits_{s \in [0,t]} \sup\limits_{x \in (0,1]} \left| \frac{J(s,x)}{x} \right|^p \leq C_{T,p}  \int_0^s \|f \|_{\mathscr{H},r}^p \textrm{d}r.
\end{equation*}
The result then follows by taking expectations.
\end{proof}
\begin{prop}\label{vol1}
Suppose that $p> 22$ and let $f \in L^p(\Omega; L^{\infty}([0,T];\mathscr{H}))$. Define $K(t,x)$ such that, for $t \in [0,T]$ and $x \in [0,1]$, 
\begin{equation*}
K(t,x):= \int_0^t \int_0^1 G(t-s,x,y)f(s,y) W(\textrm{d}y, \textrm{d}s).
\end{equation*}
Then we have that 
\begin{enumerate}
\item $K \in C([0,T];\mathscr{H})$ almost surely.
\item For $t \in [0,T]$, 
\begin{equation*}
\mathbb{E}\left[ \|K\|_{\mathscr{H},t}^p \right] \leq C_{T,p} \int_0^t \mathbb{E}\left[ \|f\|_{\mathscr{H},s}^p \right] \textrm{d}s.
\end{equation*}
\end{enumerate}
\end{prop}
\begin{proof}
In this case, we will need to apply the Garsia-Rodemich-Rumsey Lemma (in the form of Corollary A.3 in \cite{Dalang}) to obtain continuity of $\frac{K(t,x)}{x}$, and to suitably control the supremum of this process. Define for $t \in [0,T]$ and $x \in [0,1]$
\begin{equation*}
L(t,x):= \int_0^t \int_0^1 \tilde{G}(t-s,x,y) \frac{f(s,y)}{y} W(\textrm{d}y, \textrm{d}s).
\end{equation*}
Note that, for $t \in [0,T]$ and $x \in (0,1]$, we then have that $L(t,x)= K(t,x)/x$. For $x,z \in [0,1]$ and $0 \leq s \leq t \leq T$, we have
\begin{equation}\label{spliteqn}\begin{split}
\mathbb{E}\left[ |L(t,x)-L(s,z)|^p \right] \leq & \mathbb{E} \left[ \left|\int_s^t \int_0^1 \tilde{G}(t-r,x,y) \frac{f(r,y)}{y} W(\textrm{d}y, \textrm{d}r) \right|^p \;  \right] 
\\ & + \mathbb{E}\left[ \left|\int_0^s \int_0^1 \left[ \tilde{G}(t-r,x,y) - \tilde{G}(s-r,z,y)  \right] \frac{f(r,y)}{y} W(\textrm{d}y,\textrm{d}r) \right|^p \right].
\end{split}
\end{equation}
Applying Burkholder's inequality gives that the first term on the right hand side is at most
\begin{equation*}\begin{split}
\mathbb{E}  \left[ \left|\int_s^t \int_0^1 \left(\tilde{G}(t-r,x,y)  \frac{f(r,y)}{y}\right)^2 \textrm{d}y \textrm{d}r \right|^{p/2} \;  \right] \leq & \mathbb{E}\left[ \left|\int_s^t \left(\int_0^1 \tilde{G}(t-r,x,y)^2 \textrm{d}y \right) \|f\|_{\mathscr{H},r}^2 \textrm{d}r \right|^{p/2}\right] .
\end{split}
\end{equation*}
An application of H\"{o}lder's inequality then bounds this by 
\begin{equation*}
\left[\int_s^t \left(\int_0^1 \tilde{G}(t-r,x,y)^2 \textrm{d}y \right)^{p/(p-2)} \textrm{d}r \right]^{(p-2)/2} \times \int_s^t \mathbb{E}\left[\|f\|_{\mathscr{H},r}^p \textrm{d}r \right].
\end{equation*}
By estimate (2) in Proposition \ref{Heat Kernel}, this can be bounded by  
\begin{equation*}
C |t-s|^{(p-4)/4} \times \int_s^t \mathbb{E}\left[\|f\|_{\mathscr{H},r}^p \textrm{d}r \right].
\end{equation*}
Arguing similarly for the second term in (\ref{spliteqn}), making use of the other estimates (3) and (4) in Proposition \ref{Heat Kernel}, we see that 
\begin{equation*}
\mathbb{E}\left[ |L(t,x)-L(s,z)|^p \right] \leq C_{T,p} \left( |t-s|^{1/4} + |x-z|^{1/6} \right)^{p-4} \times \int_0^t \mathbb{E}\left[ \|f\|_{\mathscr{H},r}^p \right] \textrm{d}r.
\end{equation*}
In particular, for $x,z \in [0,1]$ and $0 \leq s \leq \tau \leq t \leq T$,
\begin{equation*}
\begin{split}
\mathbb{E}\left[ |L(\tau,x)-L(s,z)|^p \right] & \leq C_{T,p} \left( |\tau-s|^{1/2} + |x-z| \right)^{(p-4)/6} \times \int_0^t \mathbb{E}\left[ \|f\|_{\mathscr{H},r}^p \right] \textrm{d}r
\\ & = C_{T,p} \left( |\tau-s|^{1/2} + |x-z| \right)^{ \frac{p-22}{6} +  3} \times \int_0^t \mathbb{E}\left[ \|f\|_{\mathscr{H},r}^p \right] \textrm{d}r
\end{split}
\end{equation*}
Applying the Garsia-Rodemich-Rumsey Lemma in the form of Corollary A.3 from \cite{Dalang} we see that there exists a random variable $X_t \geq 0$ such that 
\begin{enumerate}
\item $\mathbb{E}\left[X_t^p \right] \leq C_{T,p} \int_0^t \mathbb{E}\left[ \|f\|_{\mathscr{H},r}^p \right] \textrm{d}r.$
\item For $s, \tau \in [0,t]$,
\begin{equation*}
|L(\tau,x)- L(s,y)| \leq X_t \left(|\tau-s|^{1/4} + |x-y|^{1/6} \right)^{(p-22)/p}
\end{equation*}
\end{enumerate}
In particular, setting $t = T$,  we see that $L$ is continuous on $[0,T] \times [0,1]$ almost surely, from which it follows that $K \in C([0,T];\mathscr{H})$ almost surely. In addition, we have that for $t \in [0,T]$, $\tau \in [0,t]$ and $x \in [0,1]$
\begin{equation*}
|L(\tau,x)|= |L(\tau,x)-L(0,x)| \leq C_{T,p} \times X_t.
\end{equation*}
It follows that 
\begin{equation*}
\mathbb{E}\left[ \|K\|_{\mathscr{H},t}^p \right] = \mathbb{E}\left[ \|L\|_{\infty,t}^p \right] \leq C_{T,p} \int_0^t \mathbb{E}\left[ \|f\|_{\mathscr{H},r}^p \right] \; \textrm{d}r.
\end{equation*}
\end{proof}

\section{Existence and Uniqueness}

We are now in position to prove existence and uniqueness for our Stefan problem. As mentioned in the introduction, we prove existence via a Picard argument for a truncated version of the problem. Existence for the main problem is then deduced by concatenation. Before doing so, the following truncation map is introduced.
\begin{defn}
For $M>0$, we define the map $F_M : \mathscr{H} \rightarrow \mathscr{H}$ such that 
\begin{equation*}
F_M(u)(x):= \hspace{2mm}
\begin{cases}
x \left[ \frac{u(x)}{x} \wedge M \right] \hspace{5mm} & x \in (0,1] 
\\ 0 & x =0
\end{cases}
\end{equation*}
\end{defn}
\begin{rem}
$F_M(u)$ is a well-defined map from $\mathscr{H}$ to $\mathscr{H}$ i.e. $F_M(u) \in \mathscr{H}$ for $u \in \mathscr{H}$. Continuity on $(0,1]$ is clear, and we have defined $F_M(u)(0)=0$. We therefore only need to argue the existence of a derivative for $F_M(u)$ at $x=0$. Note that, for $x >0$, 
\begin{equation*}
\frac{F_M(u)(x)}{x}= \left[ \frac{u(x)}{x} \wedge M \right].
\end{equation*}
This converges as $x \downarrow 0$, since $\frac{u(x)}{x}$ converges.
\end{rem}

\begin{thm}\label{truncate theorem}
Fix some $M>0$. Then there exists a unique solution $(v_M^1, \eta_M^1, v_M^2, \eta_M^2)$ to the system of coupled SPDEs
\begin{equation*}\label{movingboundarytruncate}\begin{split}
\frac{\partial v_M^1}{\partial t} & = \Delta v_M^1 - p^{\prime}(t) \frac{\partial (F_M(v_M^1))}{\partial x}+f_{1,M}(x, t, v_M^1(t,\cdot) )+ \sigma_{1,M}(x,t, v_M^1(t,\cdot))\dot{W} + \eta^1_M, 
\\ \frac{\partial v_M^2}{\partial t} & = \Delta v_M^2 + p^{\prime}(t) \frac{\partial (F_M(v_M^2))}{\partial x}+ f_{2,M}(x,t, v_M^2(t,\cdot) )+ \sigma_{2,M}(x,t, v_M^2(t,\cdot) )\dot{W}^- + \eta^2_M,
\\ p^{\prime}(t) & = h_M(v^1_M(t,\cdot), v^2_M(t,\cdot))
\end{split}
\end{equation*}
where $\dot{W}$ is space-time white noise with respect to the complete probability space $(\Omega, \mathscr{F}, \mathscr{F}_t, \mathbb{P})$, $v_M^1(t,0)=v_M^1(t,1)=v_M^2(t,0)=v_M^2(t,1)=0$ and initial data $(u^1_0, u^2_0) \in \mathscr{H} \times \mathscr{H}$ with $u_0^1, u_0^2 \geq 0$. The functions $f_{i,M}$, $\sigma_{i,M}$ and $h_M$ here are given by 
\begin{equation*}
f_{i,M}(x,t,u):= f_i(x,t,F_M(u)),
\end{equation*}
\begin{equation*}
\sigma_{i,M}(x,t,u):= \sigma_i(x,t,F_M(u)),
\end{equation*}
\begin{equation*}
h_M(u,v):= h(F_M(u),F_M(v)).
\end{equation*}
\end{thm}
\begin{proof}
Note that existence and uniqueness for the problem on the infinite time interval $[0,\infty)$ follow from existence and uniqueness on the finite time intervals $[0,T]$, for $T>0$. Fix some $T>0$ and some $p> 22$. We perform a Picard iteration in the space $L^p(\Omega, C([0,T];\mathscr{H})) \times L^p(\Omega, C([0,T];\mathscr{H}))$. Let $v^{i,0}(t,x)=u^i_0(x)$ for $i=1,2$. Given $(v^{1,n},v^{2,n}) \in L^p(\Omega; C([0,T];\mathscr{H})) \times L^p(\Omega; C([0,T];\mathscr{H}))$, we define $z^{1,n+1}$ and $z^{2,n+1}$ to solve the equations
\begin{equation}\label{mild trunc}\begin{split}
\frac{\partial z^{1,n+1}}{\partial t}= & \Delta z^{1,n+1} - h_M\left( v^{1,n}(t,\cdot), v^{2,n}(t,\cdot) \right)\frac{\partial (F_M(v^{1,n}))}{\partial x} \\ & + f_{1,M}(x,t,v^{1,n}(t,\cdot)) + \sigma_{1,M}(x,t,v^{1,n}(t,\cdot))\dot{W},
\end{split}
\end{equation}
and 
\begin{equation}\label{mild trunc}\begin{split}
\frac{\partial z^{2,n+1}}{\partial t}= & \Delta z^{2,n+1} + h_M\left( v^{1,n}(t,\cdot), v^{2,n}(t,\cdot) \right)\frac{\partial (F_M(v^{2,n}))}{\partial x} \\ & + f_{2,M}(x,t,v^{2,n}(t,\cdot)) + \sigma_{2,M}(x,t,v^{2,n}(t,\cdot))\dot{W}^{-},
\end{split}
\end{equation}
with initial data $u^1_0$ and $u^2_0$ respectively. We then define, for $i=1,2$,  $w^{i,n+1}$ to be given by the solution to the obstacle problem
\begin{equation*}\begin{split}
& \frac{\partial w^i_{n+1}}{\partial t}= \Delta w^i_{n+1} + \eta_{n+1}^i, \; \; \; \; \; w_{n+1}^i \geq -z_{n+1}^i, \\ & \int_0^T \int_0^1 \left[ w_{n+1}^i(t,x) + z_{n+1}^i(t,x) \; \right] \eta_{n+1}^i(\textrm{d}x,\textrm{d}t) =0, 
\end{split}
\end{equation*}
and set our $(n+1)^{\textrm{th}}$ approximate solution to be given by $(v^{1,n+1},v^{2,n+1}):= (z^{1,n+1}+w^{1,n+1}, z^{2,n+1}+w^{2,n+1})$. For ease of notation, we define here 
\begin{equation*}
g_n(t,x):= h_M\left( v^{1,n}(t,\cdot), v^{2,n}(t,\cdot) \right)F_M(v^{1,n}(t,x)).
\end{equation*} Writing $z^{1,n+1}$ in mild form, we have that 
\begin{equation*}\begin{split}
z^{1,n+1}(t,x)= & \int_0^1 G(t,x,y)u_0^1(y) \textrm{d}y \\ & + \int_0^t \int_0^1 \frac{\partial G}{\partial y}(t-s,x,y) g_n(s,y) \textrm{d}y \textrm{d}s \\ & + \int_0^t \int_0^1 G(t-s,x,y)f_{1,M}(y,s,v^{1,n}(s,\cdot)) \textrm{d}x\textrm{d}s \\ & + \int_0^t \int_0^1 G(t-s,x,y) \sigma_{1,M}(y,s,v^{1,n}(s,\cdot)) \textrm{W}(\textrm{d}y,\textrm{d}s).
\end{split}
\end{equation*}
Note that, by Propositions \ref{drift1}, \ref{drift2} and \ref{vol1}, we can see from this expression that $z^{1,n+1} \in C([0,T];\mathscr{H})$ almost surely. Using the analogous expression for $z^{1,n}$ and taking the difference, we see that
\begin{equation}\label{Picard bound}\begin{split}
\mathbb{E}& \left[ \| z^{1,n+1}- z^{1,n}\|^p_{\mathscr{H},T} \right]   \\ \leq &   \mathbb{E} \left[ \left\| \int_0^t \int_0^1 \frac{\partial G}{\partial y}(t-s,x,y) \left[ g_{n}(s,y)- g_{n-1}(s,y) \right] \textrm{d}y  \textrm{d}s \right\|^p_{\mathscr{H},T} \right] \\ & + \mathbb{E} \left[ \left\| \int_0^t \int_0^1 G(t-s,x,y)\left[ f_{1,M}(y,s,v^{1,n}(s,\cdot))- f_{1,M}(y,s,v^{1,n-1}(s,\cdot)) \right] \textrm{d}x\textrm{d}s \right\|_{\mathscr{H},T}^p \right]  \\ & + \mathbb{E}\left[ \left\| \int_0^t \int_0^1 G(t-s,x,y) \left[ \sigma_{1,M}(y,s,v^{1,n}(s,\cdot))- \sigma_{1,M}(y,s,v^{1,n-1}(s,\cdot)) \right] \textrm{W}(\textrm{d}y,\textrm{d}s) \right\|_{\mathscr{H},T}^p \right].
\end{split}
\end{equation}
Note that 
\begin{equation}\label{trunc2}
\|g_n- g_{n-1}\|_{\mathscr{H},s} \leq C_{M,h} \left( \|v^{1,n}-v^{1,n-1}\|_{\mathscr{H},s} + \|v^{2,n}-v^{2,n-1}\|_{\mathscr{H},s} \right). 
\end{equation}
By making use of Propositions \ref{drift1}, \ref{drift2} and \ref{vol1}, we are able to bound (\ref{Picard bound}). We obtain that 
\begin{equation*}
\begin{split}
\mathbb{E} \left[ \|z^{1,n+1}- z^{1,n}\|_{\mathscr{H},T}^p \right] \leq & C_{T,p} \int_0^T \mathbb{E}\left[ \|f_{1,M}(\cdot,s, v^{1,n}(s, \cdot))- f_{1,M}(\cdot,s, v^{1,n-1}(s, \cdot)) \|_{\infty}^p \right] \textrm{d}s \\ & + \int_0^T \mathbb{E}\left[\|g_n-g_{n-1}\|_{\mathscr{H},s}^p \right] \textrm{d}s \\ &  + \int_0^T  \mathbb{E}\left[\|\sigma_{1,M}(\cdot,s, v^{1,n}(s, \cdot))- \sigma_{1,M}(\cdot,s, v^{1,n-1}(s, \cdot))\|_{\mathscr{H}}^p \right] \;   \textrm{d}s.
\end{split}
\end{equation*}
By the Lipschitz-type conditions on $f, \sigma$ and the inequality (\ref{trunc2}), we see that this is at most
\begin{equation*}
C_{M,h,T,p} \int_0^T \mathbb{E}\left[\|v^{1,n}-v^{1,n-1}\|_{\mathscr{H},s}^p + \|v^{2,n}-v^{2,n-1}\|_{\mathscr{H},s}^p \right] \textrm{d}s.
\end{equation*}
By arguing in the same way, we obtain the same bound for $\mathbb{E}\left[ \|z^{2,n+1}-z^{2,n} \|_{\mathscr{H},T}^p \right].$ Adding these together gives that 
\begin{equation}\label{intbound}\begin{split}
\mathbb{E}& \left[ \|z^{1,n+1}- z^{1,n}\|_{\mathscr{H},T}^p \right] + \mathbb{E}\left[ \|z^{2,n+1}-z^{2,n} \|_{\mathscr{H},T}^p \right] \\ & \leq C_{M,h,T,p} \int_0^T \mathbb{E}\left[ \|v^{1,n}-v^{1,n-1}\|_{\mathscr{H},s}^p \right] + \mathbb{E} \left[ \|v^{2,n}-v^{2,n-1}\|_{\mathscr{H},s}^p \right] \textrm{d}s.
\end{split}
\end{equation}
By Theorem \ref{Obstacle H}, we have that for $i=1,2$,
\begin{equation*}
\|w^{i,n+1}-w^{i,n} \|_{\mathscr{H},T} \leq \|z^{i,n+1}-z^{i,n}\|_{\mathscr{H},T}
\end{equation*}
almost surely. Therefore, for $i=1,2$,
\begin{equation}\label{appic}
\|v^{i,n+1}-v^{i,n} \|_{\mathscr{H},T} \leq 2\|z^{i,n+1}-z^{i,n}\|_{\mathscr{H},T}
\end{equation}
almost surely. It follows from (\ref{intbound}) and (\ref{appic}) that 
\begin{equation*}\begin{split}
\mathbb{E}& \left[ \|v^{1,n+1}- v^{1,n}\|_{\mathscr{H},T}^p \right] + \mathbb{E}\left[ \|v^{2,n+1}-v^{2,n} \|_{\mathscr{H},T}^p \right] \\ & \leq C_{M,h,T,p} \int_0^T \mathbb{E}\left[ \|v^{1,n}-v^{1,n-1}\|_{\mathscr{H},s}^p \right] + \mathbb{E} \left[ \|v^{2,n}-v^{2,n-1}\|_{\mathscr{H},s}^p \right] \textrm{d}s.
\end{split}
\end{equation*}
By iterating this inequality, we obtain that 
\begin{equation*}\begin{split}
\mathbb{E}& \left[\|v^{1,n+1}-v^{1,n}\|_{\mathscr{H},T}^p+ \|v^{2,n+1}-v^{2,n}\|_{\mathscr{H},T}^p \right] \\ \leq & C_{M,p,T} \int_0^T \mathbb{E}\left[\|v^{1,n}-v^{1,n-1}\|_{\mathscr{H},s}^p+ \|v^{2,n}-v^{2,n-1}\|_{\mathscr{H},s}^p \right] \textrm{d}s \\ \leq & C_{M,p,T}^2 \int_0^T \int_0^s \mathbb{E}\left[\|v^{1,n-1}-v^{1,n-2}\|_{\mathscr{H},u}^p+ \|v^{2,n-1}-v^{2,n-2}\|_{\mathscr{H},u}^p \right] \textrm{d}u \;  \textrm{d}s \\ = & C_{M,p,T}^2 \int_0^T \int_u^T \mathbb{E}\left[\|v^{1,n-1}-v^{1,n-2}\|_{\mathscr{H},u}^p+ \|v^{2,n-1}-v^{2,n-2}\|_{\mathscr{H},u}^p \right] \textrm{d}s \;  \textrm{d}u \\ = & C_{M,p,T}^2 \int_0^T  \mathbb{E}\left[\|v^{1,n-1}-v^{1,n-2}\|_{\mathscr{H},u}^p+ \|v^{2,n-1}-v^{2,n-2}\|_{\mathscr{H},u}^p \right](T-u) \;  \textrm{d}u.
\end{split}
\end{equation*}
Therefore,
\begin{equation*}\begin{split}
\mathbb{E}& \left[\|v^{1,n+1}-v^{1,n}\|_{\mathscr{H},T}^p+ \|v^{2,n+1}-v^{2,n}\|_{\mathscr{H},T}^p \right] \\  & \leq C_{M,p,T}^n \int_0^T \mathbb{E}\left[ \|v^{1,1}-v^{1,0} \|_{\mathscr{H},s}^p + \|v^{2,1}-v^{2,0}\|_{\mathscr{H},s}^p \right] \frac{(T-s)^{n-1}}{(n-1)!} \textrm{d}s \\ & \leq C_{M,p,T}^n \times \mathbb{E} \left[ \|v^{1,1}-v^{1,0} \|_{\mathscr{H},s}^p + \|v^{2,1}-v^{2,0}\|_{\mathscr{H},s}^p  \right] \frac{T^{n}}{n!}.
\end{split}
\end{equation*}
Hence, for $m > n \geq 1$, we have that 
\begin{equation*}\begin{split}
\mathbb{E}& \left[\|v^{1,m}-v^{1,n}\|_{\mathscr{H},T}^p+ \|v^{2,m}-v^{2,n}\|_{\mathscr{H},T}^p \right]^{1/p} \\ & \leq \sum\limits_{k=n}^{m-1} \left[ \frac{\tilde{C}_{M,p,T}^kT^k}{k!} \right]^{1/p} \mathbb{E}\left[ \|v_1^1-v_0^1\|_{\mathscr{H},T}^p + \|v_1^2-v_0^2\|_{\mathscr{H},T}^p \right]^{1/p} \rightarrow 0.
\end{split}
\end{equation*}
as $m,n \rightarrow \infty$. Therefore, the sequence $(v^{1,n},v^{2,n})$ is Cauchy in the space $L^p(\Omega; C([0,T]; \mathscr{H})))^2$ and so converges to some pair $(v^1_M,v^2_M)$. We now verify that this is indeed a solution to our evolution equation. Let $(z^1_M,z^2_M)$ solve the SPDEs 
\begin{equation*}
\frac{\partial z_M^1}{\partial t} = \Delta z_M^1 -p^{\prime}(t)\frac{\partial (F_M(v_M^1))}{\partial x}+f_{1,M}(x, t, v_M^1 )+ \sigma_{1,M}(x,t, v_M^1)\dot{W},
\end{equation*}
\begin{equation*}
\frac{\partial z_M^2}{\partial t} = \Delta z_M^2 +p^{\prime}(t)\frac{\partial (F_M(v_M^2))}{\partial x}+f_{2,M}(x, t, v_M^2 )+ \sigma_{2,M}(x,t, v_M^2)\dot{W}^-,
\end{equation*}
where 
\begin{equation*}
p^{\prime}(t):= h_M\left( v_M^1(t, \cdot), v_M^2(t,\cdot)  \right).
\end{equation*}
We then define $(w^{i}_M, \eta_M^i)$ as solutions to the obstacle problem, with obstacles $z^i_M$, and set $\tilde{v}^i_M:= z^i_M + w^i_M$. Therefore, we have that 
\begin{equation*}
\frac{\partial \tilde{v}_M^1}{\partial t} = \Delta \tilde{v}_M^1 -p^{\prime}(t)\frac{\partial (F_M(v_M^1))}{\partial x}+f_{1,M}(x, t, v_M^1)+ \sigma_{1,M}(x,t, v_M^1)\dot{W}+ \eta^1_M,
\end{equation*}
\begin{equation*}
\frac{\partial \tilde{v}_M^2}{\partial t} = \Delta \tilde{v}_M^2 +p^{\prime}(t)\frac{\partial (F_M(v_M^2))}{\partial x}+f_{2,M}(x, t, v_M^2 )+ \sigma_{2,M}(x,t, v_M^2)\dot{W}+ \eta^2_M.
\end{equation*}
By reproducing the same argument as that used to achieve the estimate (\ref{appic}), we are able to show that 
\begin{equation*}\begin{split}
\mathbb{E}& \left[ \|\tilde{v}^1_M-v^{1,n}\|_{\mathscr{H},T}^p + \|\tilde{v}^2_M -v^{2,n}\|_{\mathscr{H},T}^p \right] \\ & \leq  C_{M,h,T,p} \int_0^T \mathbb{E}\left[ \|v^1_M- v^{1,n}\|_{\mathscr{H},s}^p + \|v^2_M- v^{2,n} \|_{\mathscr{H},s}^p \right] \textrm{d}s \rightarrow 0.
\end{split}
\end{equation*}
It follows that $\tilde{v}_M^1= v^1_M$ in $L^p(\Omega; C([0,T];\mathscr{H}))$. The same applies to $v^2_M$, so it follows that the pair $(v^1_M,v^2_M)$, together with the reflection measures $(\eta^1_M, \eta^2_M)$, do indeed satisfy our problem. 

Uniqueness follows by essentially the same argument. Given two solutions with the same initial data, $(v^1_1,v^2_1)$ and $(v_2^1,v_2^2)$ (together with their reflection measures), we argue as before to obtain that, for $t \in [0,T]$, 
\begin{equation*}
\mathbb{E}\left[\|v^1_1-v_2^1\|_{\mathscr{H},t}^p+ \|v^2_1-v^2_2\|_{\mathscr{H},t}^p \right] \leq \int_0^t \mathbb{E}\left[\|v^1_1-v_2^1\|_{\mathscr{H},s}^p+ \|v^2_1-v^2_2\|_{\mathscr{H},s}^p \right] \textrm{d}s.
\end{equation*}
The equivalence then follows by Gronwall's inequality.
\end{proof}
Before stating and proving our main theorem, we first prove the following technical result which gives us a universal representation for terminal times of maximal solutions, and a canonical localising sequence.
\begin{prop}\label{local}
Let $(v^1, \eta^1, v^2,\eta^2)$ be a maximal solution to the moving boundary problem in the relative frame. Then we have that the terminal time, $\tau$, is given by 
\begin{equation*}
\tau:= \sup\limits_{M >0} \tau_M,
\end{equation*}
where 
\begin{equation*}
\tau_M:= \inf\left\{ t\geq 0 \; | \; \|v^1\|_{\mathscr{H},t} + \|v^2\|_{\mathscr{H},s} \geq M \right\}.
\end{equation*}
Furthermore, $\tau_M$ can be taken as a localising sequence for the solution.
\end{prop}
\begin{proof}
Let $\sigma_n$ be a localising sequence for the solution. Define $\sigma:= \sup_{n \geq 1} \sigma_n$. Fix some $M>0$ and consider the localising sequence $(\sigma_n \wedge \tau_M)_{n \geq 1}$ for the solution, with this solution on a potentially smaller stochastic interval. Taking the limit as $n \rightarrow \infty$ for each $M>0$, we obtain a local solution which agrees with $(v^1, \eta^1, v^2,\eta^2)$ until $\tau \wedge \sigma$ (its own terminal time), and has the localising sequence $(\sigma \wedge \tau_M)_{M >0}$. Arguing as in Theorem \ref{truncate theorem}, for every $M>0$, this must agree with the $M$-truncated problem until the time $\sigma \wedge \tau_M$. Therefore, for every $M>0$, we have that $\sigma \geq \tau_M$, as otherwise we could use the solution to the $M$-truncated problem to propagate the solution to a later time. Hence, the localising sequence $\sigma \wedge \tau_M$ is simply just $\tau_M$. In addition, we have that $\sigma \geq \tau_M$ for every $M$ almost surely, which implies that $\sigma \geq \tau$ almost surely. Clearly, by agreement with the $M$-truncated problems until the times $\tau_M$, we have that $\tau \geq \sigma$, since the solution cannot be propagated beyond a blow-up in the $\mathscr{H}$-norm. We therefore have the result.
\end{proof}

We are now in position to prove our main result, Theorem \ref{main}.
\begin{proof}[Proof of Theorem \ref{main}]
We prove existence of a maximal solution by concatenating solutions to the truncated problems. For $M>0$, let $\mathscr{V}_M:=(v^1_M, \eta^1_M, v^2_M, \eta^2_M)$ solve the $M$-truncated problem as in Theorem \ref{truncate theorem}. Note that for $M_1\geq M_2>0$, $\mathscr{V}_{M_1}$ solves (\ref{movingboundary}) until the random stopping time 
\begin{equation*}
\tau:= \inf \left\{ t \in [0,\infty) \; | \; \max\left\{\|v^1_{M_1}\|_{\mathscr{H},t}, \|v^2_{M_1}\|_{\mathscr{H},t} \right\} \geq M_2 \right\}. 
\end{equation*}
Let $z^i_{M_j}$ solve the unreflected SPDEs associated with $v^i_{M_j}$ for $i,j=1,2$, as in the equation (\ref{mild trunc}). By considering the difference between $z^i_{M_1}(t,x) \mathbbm{1}_{\left\{t \leq \tau \right\}}$ and $z^i_{M_2}(t,x) \mathbbm{1}_{\left\{t \leq \tau \right\}}$ for $i=1,2$, we are able to argue as in the proof of Theorem \ref{truncate theorem} to obtain that, for $T>0$, $p> 22$ and $t \in [0,T]$
\begin{equation*}\begin{split}
\mathbb{E}& \left[ \|z^1_{M_1}- z^1_{M_2} \|_{\mathscr{H}, \tau\wedge t}^p + \|z^2_{M_1}- z^2_{M_2} \|_{\mathscr{H}, \tau\wedge t}^p \right] \\ & \leq C_{T,p} \int_0^t \mathbb{E}\left[  \|z^1_{M_1}- z^1_{M_2} \|_{\mathscr{H}, \tau\wedge s}^p + \|z^2_{M_1}- z^2_{M_2} \|_{\mathscr{H}, \tau\wedge s}^p \right] \textrm{d}s.
\end{split}
\end{equation*}
Gronwall's lemma then applies to give that $z^i_{M_1}=z^i_{M_2}$ almost surely for $i=1,2$, until the random time $\tau \wedge T$. Since $(v^i_{M_j}, \eta^i_{M_j})$ are simply the solutions to the obstacle problem with obstacles $z^i_{M_j}$, it follows that $(v^1_{M_1}, \eta^1_{M_1},v^2_{M_1},\eta^2_{M_1})=(v^1_{M_2}, \eta^1_{M_2},v^2_{M_2},\eta^2_{M_2})$ until time $\tau \wedge T$.  This holds for every $T>0$, so we have that they almost surely agree until time $\tau $. We have therefore shown that the solutions to our $M$-truncated problems are consistent. This allows us to concatenate them. We define the process $(v^1,\eta^1, v^2, \eta^2)$ such that, for all $M>0$, it agrees with $(v^1_{M}, \eta^1_{M},v^2_{M},\eta^2_{M})$ until the time 
\begin{equation*}
\tau_M := \inf \left\{ t \in [0,\infty) \; | \; \max\left\{\|v^1_{M}\|_{\mathscr{H},t}, \|v^2_{M}\|_{\mathscr{H},t} \right\} \geq M \right\}.
\end{equation*}
Note that 
\begin{equation*}
\tau_M=\inf \left\{ t \in [0,\infty) \; | \; \max\left\{\|v^1\|_{\mathscr{H},t}, \|v^2\|_{\mathscr{H},t} \right\} \geq M \right\}.
\end{equation*}
Since $\mathscr{V}_M$ solves (\ref{movingboundary}) until time $\tau_M$, we have that  $(v^1,\eta^1, v^2, \eta^2)$ solves (\ref{movingboundary}) on the random interval $[0, \tilde{\tau})$, where
\begin{equation*}
\tilde{\tau}= \sup\limits_{M >0} \tau_M,
\end{equation*}
with localising sequence $\tau_M$. This solution is maximal since, on the set $\left\{ \tilde{\tau} < \infty \right\}$, we have that 
\begin{equation*}
\max\left\{\|v^1\|_{\mathscr{H},t}, \|v^2\|_{\mathscr{H},t} \right\} \rightarrow \infty
\end{equation*}
as $t \uparrow \tilde{\tau}$ almost surely. Therefore, we know that a maximal solution exists. Uniqueness follows by the same arguments as those made for consistency among solutions to the truncated problems. Fix $T>0$ and let $(\tilde{v}^1, \tilde{\eta}^1, \tilde{v}^2, \tilde{\eta}^2)$ be a maximal solution to (\ref{movingboundary}). Note that the localising sequence can be taken as in Proposition \ref{local}. We then have that for every $M>0$, it must agree with the solution to the $M$-truncated problem, and therefore with $(v^1,\eta^1, v^2, \eta^2)$, until 
\begin{equation*}
\inf \left\{ t \in [0,T] \; | \;  \max \left\{\|\tilde{v}^1\|_{\mathscr{H},t}, \|\tilde{v}^2\|_{\mathscr{H},t} \right\} \geq M \right\}.  
\end{equation*}
Uniqueness then follows.
\end{proof}
We can now show that, as one might expect, blow up for the problem occurs precisely when the boundary speed blows up.
\begin{prop}\label{der blow up}
Let $(v^1, \eta^1, v^2, \eta^2)$ be a solution to the moving boundary problem in the relative frame on the maximal interval $[0,\tau)$. Then 
\begin{equation*}
\tau= \sup\limits_{M>0} \sigma_M,
\end{equation*}
where 
\begin{equation*}
\sigma_M=\inf \left\{ t \in [0,\infty) \; | \; h(v^1(t,\cdot), v^2(t,\cdot)) \geq M \right\}.
\end{equation*}
\begin{proof}
By Proposition \ref{local}, and since $h$ is bounded on bounded sets in $\mathscr{H} \times \mathscr{H}$, it is sufficient to show that for every $M>0$, $\sigma_M \leq \tau$ almost surely. Suppose that, for every $T>0$,
\begin{equation}\label{condition}
\mathbb{E}\left[ \sup\limits_{t \in [0, (\sigma_M \wedge \tau \wedge T))} \left[ \|v^1(t,\cdot) \|_{\mathscr{H}}^p +  \|v^2(t,\cdot) \|_{\mathscr{H}}^p \right] \right] < \infty.
\end{equation}
Then, since
\begin{equation*}
\sup\limits_{ t \in [0,\tau \wedge T)} \left[ \|v^1(t,\cdot) \|_{\mathscr{H}} +  \|v^2(t,\cdot) \|_{\mathscr{H}} \right] = \infty
\end{equation*}
almost surely on the set that $\left\{ \tau \leq T \right\}$ we must have that $\sigma_M < \tau \wedge T$ almost surely on the set $\left\{ \tau \leq T \right\}$. Therefore, for every $T>0$, we would have that 
\begin{equation*}
\sigma_M \mathbbm{1}_{\left\{ \tau \leq T \right\}} < \tau \wedge T
\end{equation*}
almost surely. Letting $T \uparrow \infty$ (via a countable sequence), we obtain that 
\begin{equation*}
\sigma_M\mathbbm{1}_{\left\{ \tau < \infty \right\}} \leq \tau
\end{equation*} almost surely, so that $\sigma_M \leq \tau$ almost surely, giving the result. So it is sufficient to prove that (\ref{condition}) holds for every $M,T>0$. Fix some $T>0$ and some $M>0$. By Theorem \ref{Obstacle H},
\begin{equation*}
\mathbb{E}\left[ \sup\limits_{ t \in [0,(\sigma_M \wedge \tau\wedge T))} \|v^1(t,\cdot) \|_{\mathscr{H}}^p  \right] \leq 2^p \mathbb{E}\left[ \sup\limits_{ t \in [0,(\sigma_M \wedge \tau \wedge T))} \|z^1(t,\cdot) \|_{\mathscr{H}}^p  \right],
\end{equation*}
where $z^1$ solves the unreflected problem corresponding to $v^1$ on the random interval $[0,\tau \wedge T)$. That is, $z^1$ solves 
\begin{equation*}\begin{split}
\frac{\partial z^{1}}{\partial t}=  \Delta z^{1} - h(v^1(t,\cdot), v^2(t,\cdot))\frac{\partial v^{1}}{\partial x} + f(x,t, v^{1}(t,\cdot)) + \sigma(x,t,v^{1}(t,\cdot))\dot{W}(\textrm{d}x, \textrm{d}t),
\end{split}
\end{equation*}
on the interval $[0,\tau \wedge T)$, where we take the localising sequence $$\tau_N:= \inf \left\{ t \in [0,\infty) \; | \; \max\left\{ \|v^1\|_{\mathscr{H}, t}+ \|v^2\|_{\mathscr{H}, t} \right\} \geq N \right\}.$$ Writing $z^1$ in mild form, we have that for $t \in [0,T]$
\begin{equation}\label{longeqn}\begin{split}
z^1 (t,x)&   \mathbbm{1}_{\left\{ t < \sigma_M \wedge \tau_N \wedge T \right\} }= \\  &  \mathbbm{1}_{\left\{ t < \sigma_M \wedge \tau _N \wedge T \right\} } \left[\int_0^1 G(t,x,y)u_0^1(y) \textrm{d}y \right] + \\ & \mathbbm{1}_{\left\{ t < \sigma_M \wedge \tau_N \wedge T \right\} } \left[ \int_0^t \int_0^1 \frac{\partial G}{\partial y}(t-s,x,y)h(v^1(s,\cdot), v^2(s,\cdot)) v^1(s,y)   \textrm{d}y \textrm{d}s \right] +  \\ & \mathbbm{1}_{\left\{ t < \sigma_M \wedge \tau_N \wedge T \right\} } \left[ \int_0^t \int_0^1 G(t-s,x,y)f(y,s, v^{1}(s,\cdot)) \textrm{d}x\textrm{d}s \right] + \\ & \mathbbm{1}_{\left\{ t < \sigma_M \wedge \tau_N \wedge T\right\} } \left[  \int_0^t \int_0^1 G(t-s,x,y) \sigma(y,s, v^{1}(s,\cdot)) \textrm{W}(\textrm{d}y,\textrm{d}s)\right],
\end{split}
\end{equation}
We aim to bound the $L^p(\Omega ; L^{\infty}([0,T]; \mathscr{H}))$ norms of the four terms on the right hand side uniformly over $N$. The first term is simple, so we omit the proof for this. For the second term, note that 
\begin{equation*}\begin{split}
\mathbbm{1}_{\left\{ t < \sigma_M \wedge \tau_N \wedge T \right\} } &  \left[ \int_0^t \int_0^1 \frac{\partial G}{\partial y}(t-s,x,y)h(v^1(s,\cdot), v^2(s,\cdot)) v^1(s,y)   \textrm{d}y \textrm{d}s \right] \\ & \leq   \left[ \int_0^t \int_0^1 \frac{\partial G}{\partial y}(t-s,x,y)h(v^1(s,\cdot), v^2(s,\cdot))  v^1(s,y)\mathbbm{1}_{\left\{ s < \sigma_M \wedge \tau_N \wedge T \right\} }   \textrm{d}y \textrm{d}s \right].
\end{split}
\end{equation*} 
By the definition of $\sigma_M$, this is at most 
\begin{equation*}
J(t,x):=C_{M} \left[ \int_0^t \int_0^1  \frac{\partial G}{\partial y}(t-s,x,y)  v^1(s,y)\mathbbm{1}_{\left\{ s < \sigma_M \wedge \tau_N \wedge T  \right\} }   \textrm{d}y \textrm{d}s \right].
\end{equation*}
An application of Proposition \ref{drift2} then gives that 
\begin{equation*}
\mathbb{E}\left[ \left\| J \right\|_{\mathscr{H},t}^p \right] \leq C_{M,T,p} \int_0^t \mathbb{E}\left[ \sup\limits_{r \in [0, s]}\|v^1(r,\cdot) \mathbbm{1}_{\left\{ r < \sigma_M \wedge \tau_N \wedge T \right\} }\|_{\mathscr{H}}^p  \right] \textrm{d}s.
\end{equation*}
Arguing in the same way, applying Proposition \ref{drift1} in place of Proposition \ref{drift2}, we are able to control the third term of (\ref{longeqn}). For the fourth term of (\ref{longeqn}), we  note that
\begin{equation*}\begin{split}
\mathbbm{1}_{\left\{ t < \sigma_M \wedge \tau_N \wedge T \right\} } & \left| \int_0^t \int_0^1 G(t-s,x,y) \sigma(y,s, v^{1}(s,\cdot)) \textrm{W}(\textrm{d}y,\textrm{d}s)\right| \\ & \leq \left| \int_0^t \int_0^1 G(t-s,x,y) \sigma(y,s, v^{1}(s,\cdot))\mathbbm{1}_{\left\{ s < \sigma_M \wedge \tau_N \wedge T \right\} } \textrm{W}(\textrm{d}y,\textrm{d}s)\right| =: K(t,x).
\end{split}
\end{equation*}
Applying Proposition \ref{vol1} and noting the linear growth condition on $\sigma$ then gives that, for $t \in [0,T]$ 
 \begin{equation*}
\mathbb{E}\left[ \|K\|_{\mathscr{H},t}^p \right] \leq C_{T,p}\left[ 1+  \int_0^t \mathbb{E}\left[ \sup\limits_{r \in [0, s]}\|v^1(r,\cdot) \mathbbm{1}_{\left\{ r < \sigma_M \wedge \tau_N \wedge T \right\} }\|_{\mathscr{H}}^p  \right] \textrm{d}s \right].
\end{equation*}
Putting the four terms together, we have that, for $ t \in [0,T]$, 
\begin{equation*}
\mathbb{E}\left[ \sup\limits_{r \in [0,t]} \left\|z^1(r,\cdot) \mathbbm{1}_{\left\{ r < \sigma_M \wedge \tau_N \wedge T \right\} }\right\|_{\mathscr{H}}^p \right] \leq C_{T,p,M} \left[ 1+  \int_0^t \mathbb{E}\left[ \sup\limits_{r \in [0, s]}\|v^1(r,\cdot) \mathbbm{1}_{\left\{ r < \sigma_M \wedge \tau_N \wedge T \right\} }\|_{\mathscr{H}}^p  \right] \textrm{d}s \right].
\end{equation*}
Therefore, 
\begin{equation*}
\mathbb{E}\left[ \sup\limits_{r \in [0,t]} \left\|v^1(r,\cdot) \mathbbm{1}_{\left\{ r < \sigma_M \wedge \tau_N \wedge T \right\} }\right\|_{\mathscr{H}}^p \right] \leq C_{T,p,M} \left[ 1+  \int_0^t \mathbb{E}\left[ \sup\limits_{r \in [0, s]}\|v^1(r,\cdot) \mathbbm{1}_{\left\{ r < \sigma_M \wedge \tau_N \wedge T \right\} }\|_{\mathscr{H}}^p  \right] \textrm{d}s \right].
\end{equation*}
It follows by Gronwall's inequality that 
\begin{equation*}
\sup\limits_{N >0} \mathbb{E}\left[ \sup\limits_{t \in [0,\sigma_M \wedge \tau_N \wedge T)} \left\|v^1(t,\cdot)\right\|_{\mathscr{H}}^p \right]= \sup\limits_{N >0} \mathbb{E}\left[ \sup\limits_{t \in [0,T]} \left\|v^1(t,\cdot) \mathbbm{1}_{\left\{ t < \sigma_M \wedge \tau_N \wedge T \right\} }\right\|_{\mathscr{H}}^p \right] \leq K_{T,p,M} < \infty.
\end{equation*}
Note that, by the MCT
\begin{equation*}
\sup\limits_{N >0} \mathbb{E}\left[ \sup\limits_{t \in [0,\sigma_M \wedge \tau_N \wedge T)} \left\|v^1(t,\cdot)\right\|_{\mathscr{H}}^p \right] = \mathbb{E}\left[ \sup\limits_{t \in [0,\sigma_M \wedge \tau \wedge T)} \left\|v^1(t,\cdot)\right\|_{\mathscr{H}}^p \right]
\end{equation*}
The same estimates hold for $v^2$. This concludes the proof.
\end{proof}
\end{prop}
\begin{rem}
The previous result gives that solutions to the moving boundary problem can only blow-up in ways which cause the boundary speed to blow up. For example, if we have the classical boundary behaviour given by $$h(u,v) = \gamma( u^{\prime}(0) - v^{\prime}(0)),$$ for a constant, $\gamma$, we know that blow-up can only occur when one of the spatial derivatives at the shared interface blow up.
\end{rem}

\begin{cor}
Suppose that the function $h$ is globally bounded. Then the unique maximal solution to (\ref{movingboundary}) is global i.e. $\tau=\infty$ almost surely.
\end{cor}

\section{Numerical Illustrations}

In this section, we implement a numerical scheme in order to illustrate some typical profiles for solutions to our Stefan problems and point out some of the features which were highlighted in the previous analysis. In particular, we are able to see the presence of spatial derivatives at the shared boundary when we choose the drift and volatility parameters appropriately. We will also contrast this with the behaviour at the boundary when the volatility only decays there like $\sqrt{x}$, providing numerical evidence that a spatial derivative does not exist at the boundary in this case. This is of interest as we have not shown the linear decay condition which we imposed on the volatility to be optimal. To begin, we briefly describe the numerical scheme used. We will consider drift and volatility functions which only depend on the position in space (in the relative frame) and the value of the solution at that particular position at that particular time. We also choose these functions to be the same for both sides of the Stefan problem. In addition, our functions $h$ which determine the boundary behaviour will take the classical form 
\begin{equation*}
h(u,v)= \gamma \left( u^{\prime}(0)- v^{\prime}(0) \right),
\end{equation*}
where $\gamma$ is a positive constant. We simulate the process on the interval $[0,T]$ and space interval $[p(t)-1,p(t)+1]$ at a given $t \in [0,T]$. That is, we simulate the equation in the relative frame on $[0,T] \times [-1,1]$, with $v^1, v^2$ supported on $[0,T] \times [0,1]$. We discretise the equation into $M$ time steps and $N$ space steps, and define $t_i:= iT/M$, $x_i:= i/N$. Given our simulated solution up to the $j^{\textrm{th}}$ time step, we define the values of the approximate solution at the $(j+1)^{\textrm{th}}$ time step by setting
\begin{equation*}\begin{split}
v^1(t_{j+1},x_i):= & \left|v^1(t_j,x_i)+ \frac{N^2}{M}\left( v^1(t_j,x_{i+1})+ v^1(t_j,x_{i-1})-2v^1(t_j,x_i) \right) \right. \\ & \left. - \gamma \frac{T N^2}{M}\left(v^1(t_j,x_1)-v^2(t_j,x_1) \right)\left(v^1(t_j,x_{i+1})-v(t_j,x_i) \right) \right. \\ & \left. + \frac{T}{M}f(x_i,v^1(t_j,x_i)) + \frac{\sqrt{TN}}{\sqrt{M}}\sigma(x_i,v^1(t_j,x_i))Z_{i,j}   \right|,
\end{split}
\end{equation*}
where the $Z_{i,j}$ are independent unit normal random variables. $v^2(t_{j+1},x_i)$ is defined similarly. We note that taking the absolute value on the right hand side here is intended to capture the fact that the process is reflected at zero. Given our simulated processes $v^1$ and $v^2$, we are then able to reproduce the boundary process by noting that $p^{\prime}(t)=h(v^1(t,\cdot),v^2(t,\cdot))$. We can therefore take 
\begin{equation*}
p(t_j):= \gamma \sum\limits_{k=1}^j \frac{T N}{M}\left(v^1(t_k,x_1)-v^2(t_k,x_1) \right). 
\end{equation*}
As we are interested in the effects of different volatility functions on our equations, we will fix the other parameters here. Throughout all of our simulations, we will choose the drift function $f$ and the parameter $\gamma$ such that $f=1$ and $\gamma = 10$. The initial data for both $v^1$ and $v^2$ will be given by the function $u_0$, where $u_0(0)=u_0(1)=0$, and $u_0(0.5)=1$, with linear interpolation in between these points. The volatility functions for which we perform our simulations are given by $\sigma_a(x),\sigma_b(u)$ and $\sigma_c(x)$, where
\begin{enumerate}
\item $\sigma_a(x)$ is linear between the points $0$, $0.5$ and $1$, with $\sigma_a(0)=0$, $\sigma_a(0.5)=\sigma_a(1)=1$.
\item $\sigma_b(u)=u$.
\item $\sigma_c(x)= \sqrt{x}$.
\end{enumerate}
The following figures depict the results of our simulations.

\begin{figure}[H]
        \centering
        \begin{subfigure}[b]{0.4\textwidth}
            \centering
            \includegraphics[width=\textwidth]{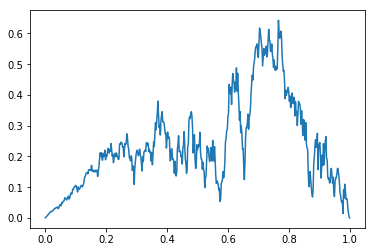}
            \caption[]%
            {{}}    
        \end{subfigure}
        \quad
        \begin{subfigure}[b]{0.4\textwidth}   
            \centering 
            \includegraphics[width=\textwidth]{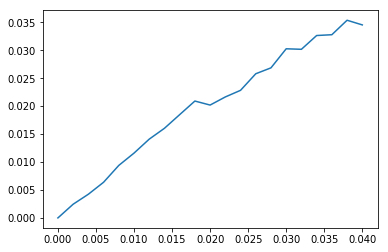}
            \caption[]%
            {{}}    
        \end{subfigure}
        \caption[]
        {\footnotesize Static Snapshot of $v^1$ at time t=0.1 on [0,1] and on [0,0.04], in the case where  $\sigma=\sigma_a(x)$, so that $\sigma$ decays linearly at the boundary.} 
\end{figure}

\begin{figure}[H]
        \centering
        \begin{subfigure}[b]{0.4\textwidth}
            \centering
            \includegraphics[width=\textwidth]{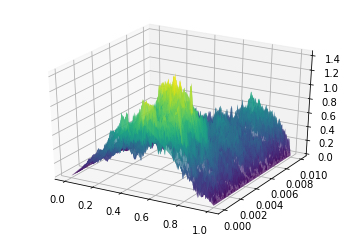}
            \caption[]%
            {{}}    
        \end{subfigure}
        \quad
        \begin{subfigure}[b]{0.4\textwidth}   
            \centering 
            \includegraphics[width=\textwidth]{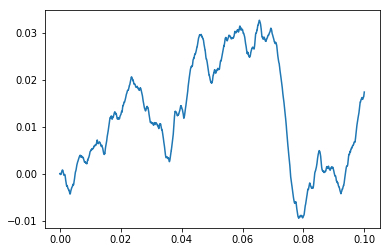}
            \caption[]%
            {{}}    
        \end{subfigure}
        \caption[]
        {\footnotesize 3D plot of $v^1$ and plot of the boundary motion, $p(t)$, in the case where $\sigma=\sigma_a(x)$, so that $\sigma$ decays linearly at the boundary.} 
\end{figure}

We see in Figure 1 that the profile is less volatile as the boundary $x=0$ is approached, and the derivative at the boundary is visible. The boundary function appears to be smooth to some degree in Figure 2 (b), as expected.  Our second set of simulations deals with the case when the volatility is multiplicative, which also falls within the framework of our earlier analysis. Figures 3 and 4 show the results of our simulation in this case. 

\begin{figure}[H]
        \centering
        \begin{subfigure}[b]{0.4\textwidth}
            \centering
            \includegraphics[width=\textwidth]{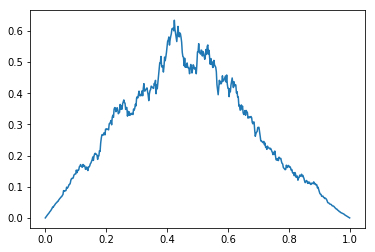}
            \caption[]%
            {{}}    
        \end{subfigure}
        \quad
        \begin{subfigure}[b]{0.4\textwidth}   
            \centering 
            \includegraphics[width=\textwidth]{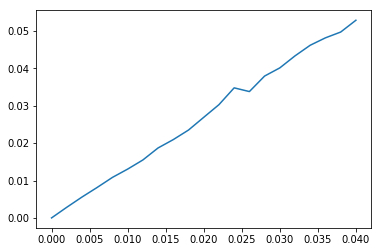}
            \caption[]%
            {{}}    
        \end{subfigure}
        \caption[]
        {\footnotesize Static Snapshot of $v^1$ at time t=0.1 on [0,1] and on [0,0.04], in the case where $\sigma=\sigma_b(u)=u$.} 
\end{figure}

\begin{figure}[H]
        \centering
        \begin{subfigure}[b]{0.4\textwidth}
            \centering
            \includegraphics[width=\textwidth]{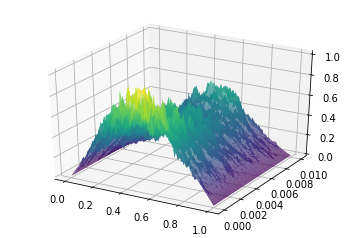}
            \caption[]%
            {{}}    
        \end{subfigure}
        \quad
        \begin{subfigure}[b]{0.4\textwidth}   
            \centering 
            \includegraphics[width=\textwidth]{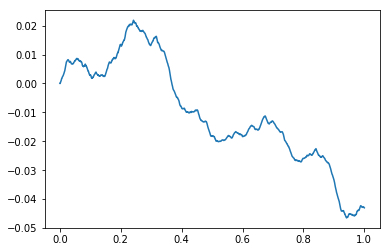}
            \caption[]%
            {{}}    
        \end{subfigure}
        \caption[]
        {\footnotesize 3D plot of $v^1$ and plot of the boundary motion, $p(t)$, in the case where $\sigma=\sigma_b(u)=u$.} 
\end{figure}

The presence of spatial derivatives at both ends of the profile is visible in Figure 3 (a). This is simply because the equation is symmetric. The Dirichlet condition is imposed at both $0$ and $1$, so that the volatility decays sufficiently quickly at both ends.

We now present the case where the volatility decays like $\sqrt{x}$ at the shared boundary, which falls outside of our proof for existence and uniqueness. 
\begin{figure}[H]
        \centering
        \begin{subfigure}[b]{0.4\textwidth}
            \centering
            \includegraphics[width=\textwidth]{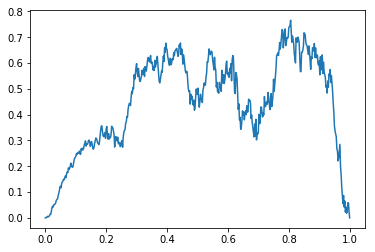}
            \caption[]%
            {{}}    
        \end{subfigure}
        \quad
        \begin{subfigure}[b]{0.4\textwidth}   
            \centering 
            \includegraphics[width=\textwidth]{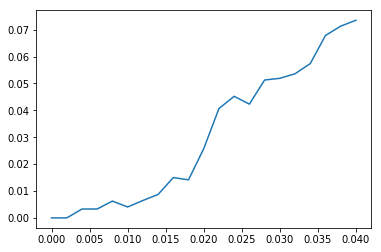}
            \caption[]%
            {{}}    
        \end{subfigure}
        \caption[]
        {\footnotesize Static Snapshot of $v^1$ at time t=0.1 on [0,1] and on [0,0.04], in the case where $\sigma=\sigma_c(x)=\sqrt{x}$.} 
\end{figure}

\begin{figure}[H]
        \centering
        \begin{subfigure}[b]{0.4\textwidth}
            \centering
            \includegraphics[width=\textwidth]{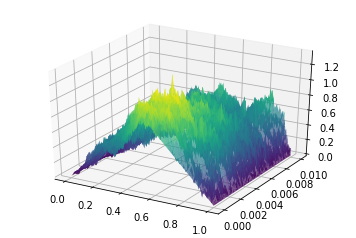}
            \caption[]%
            {{}}    
        \end{subfigure}
        \quad
        \begin{subfigure}[b]{0.4\textwidth}   
            \centering 
            \includegraphics[width=\textwidth]{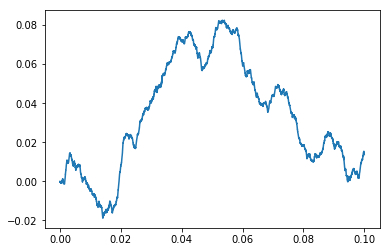}
            \caption[]%
            {{}}    
        \end{subfigure}
        \caption[]
        {\footnotesize 3D plot of $v^1$ and plot of the boundary motion in the case where $\sigma=\sigma_c(x)=\sqrt{x}$.} 
\end{figure}

In this case, the profiles can be seen to be significantly rougher close to the boundary. They do not appear to have spatial derivatives there (although this is inconclusive of course). Interestingly, the boundary motion appears to be rough to some degree. We note, however, that the scheme did not blow up- we were able to obtain sensible profiles and sensible processes for the movement of the boundary. This suggests that, perhaps by considering the derivative at the boundary in a suitable weak sense, equations of this form may be solvable without the need for linear decay of the volatility, and could produce rough paths for the boundary motion. 

We perform a simple numerical test in order to give some (albeit naive) quantification of how linear the profiles are close to the boundary in the cases presented above. For a profile at a given time, we fit the values in the spatial interval $[0,0.04]$ to a line passing through the origin by the method of least squares. Calculating the sum of the squared residuals and dividing by the square of height of the fitted line at position 0.04, we obtain values indicating how well the profile can be fitted to a line, with these values now independent of scale and so comparable across different profiles. Time averaging these values for each of the three cases, we obtain the following values, which clearly indicate that there is significant deviation from being linear close to zero when $\sigma=\sigma_c(x)=\sqrt{x}$.

\begin{table}[!h]
\centering 
\begin{tabular}{ccccccc}
\hline\hline
$\sigma_a(x)$ & $\sigma_b(u)$ & $\sigma_c(x)$ \\ \hline
0.04274 & 0.01473 & 0.28378 \\
\hline
\end{tabular}
\label{tab:hresult}
\caption{Time-Averaged measure of deviation of profile from linear fit on the spatial interval $[0,0.04]$.} 
\end{table}

\textbf{Acknowledgements.} 
The research of J. Kalsi was supported by EPSRC (EP/L015811/1).

\appendix

\section{Proof of Proposition \ref{Heat Kernel}}

We provide the main details for the proofs of the inequalities stated in Proposition 4.2. Recalling Remark \ref{G}, we focus on proving the estimates for the component $J$ of $G$, where $J$ is given by
\begin{equation*}
J(t,x,y)= \frac{1}{\sqrt{4 \pi t}} \left[ \exp\left( - \frac{(x-y)^2}{4t} \right) - \exp\left(- \frac{(x+y)^2}{4t} \right) \right].
\end{equation*}
We also define $\tilde{J}$ such that 
\begin{equation}\label{H}
\tilde{J}(t,x,y):= \frac{y}{x}J(t,x,y).
\end{equation}
The proofs of the estimates in Lemma 3.2 of \cite{Zheng} were very helpful for producing the following arguments.
\subsection{Proof of Inequality 1}
\begin{proof}
The first inequality in the Proposition states that, for every $T>0$, there exists a constant $C_T$ such that, for every $t \in (0,T]$,
\begin{equation*}
 \sup\limits_{x \in (0,1]} \int_0^1  \frac{1}{x} G(t,x,y)  \textrm{d}y \leq \frac{C_T}{\sqrt{t}}.
\end{equation*}
We prove the corresponding bound for $J$. 
\begin{equation*}\begin{split}
\frac{1}{x} J(t,x,y)= & \frac{1}{x \sqrt{4 \pi t}} \left[ \exp\left( - (x-y)^2/4t \right) - \exp\left(- (x+y)^2/4t \right) \right] \\ = & \frac{1}{x \sqrt{4 \pi t}}\exp\left(-(x-y)^2/4t \right) \left[ 1- \exp\left( - xy/t \right) \right].
\end{split}
\end{equation*}
Note that, since $x \geq 0$, 
\begin{equation*}
xy= x( (y-x)+x) \leq 2x^2 \mathbbm{1}_{\left\{ y-x \leq x \right\}}+ 2x(y-x) \mathbbm{1}_{\left\{ y-x > x \right\}}.
\end{equation*}
Therefore
\begin{equation*}
\left( 1- \exp(-xy/t)\right) \leq \left( 1-\exp(-2x^2/t) \right) \mathbbm{1}_{\left\{ y-x \leq x \right\}} +\left(1-  \exp(-2x(y-x)/t) \right) \mathbbm{1}_{\left\{ y-x > x \right\}}.
\end{equation*}
Since, for $x \geq 0$
\begin{equation*}
(1-e^{-x}) \leq \min(x, \sqrt{x}),
\end{equation*} we have that 
\begin{equation*}
\left( 1- \exp(-xy/t)\right) \leq 2 \left[ \frac{x}{\sqrt{t}} + \frac{x|y-x|}{t} \right].
\end{equation*}
Putting these bounds together, we see that
\begin{equation*}\begin{split}
\int_0^1 \frac{1}{x} J(t,x,y) \; \textrm{d}y & \leq C \int_0^1 \frac{1}{\sqrt{t}} \exp \left(- \frac{(x-y)^2}{4t} \right) \left[ \frac{1}{\sqrt{t}} + \frac{|y-x|}{t} \right] \textrm{d}y \\ & \leq C \int_{\mathbb{R}} \left[ \frac{1}{t}e^{-y^2/4t} + \frac{|y|}{t^{3/2}} e^{-y^2/4t} \right] \textrm{d}y = \frac{C}{\sqrt{t}}.
\end{split}
\end{equation*}
\end{proof}

\subsection{Proof of Inequality 3}

\begin{proof}
We begin by recalling the inequality. For every $T>0$ and $q \in (1,2)$, we claim that $\exists C_{T,q}$ such that 
\begin{equation*}
\sup\limits_{t \in [0,T]} \int_0^t \left[ \int_0^1 \left( \tilde{G}(s,x,z)- \tilde{G}(s,y,z) \right)^2 \textrm{d}z \right]^q \textrm{d}s \leq C_{T,q}|x-y|^{(2-q)/3}.
\end{equation*}
We will prove the corresponding bound for $\tilde{J}$. The arguments made in the proof of Lemma 3.2.1 of \cite{Zheng} are followed here, and we make adjustments where necessary. Assume wlog that $x \leq y$, and define $h=y-x$. By following the proofs of inequalities (1) and (2) in Lemma 3.2.1 of \cite{Zheng}, we arrive at the inequality
\begin{equation*}
\begin{split}
0 \leq \int_0^1 \left( \tilde{J}(s,x,z)- \tilde{J}(s,y,z) \right)^2 \textrm{d}z  \leq & \frac{C}{\sqrt{s}} \left[ 1+ s \left( \frac{1-e^{-x^2/4s}}{x^2} + \frac{1-e^{-(x+h)^2/4s}}{(x+h)^2} \right) \right. \\ & - e^{-h^2/16s} \left. \left(1+ 2s \frac{1-e^{-\frac{x(x+h)}{4s}}}{x(x+h)} \right)  \right]=: CR(s,x,h).
\end{split}
\end{equation*}
Therefore, it is enough to prove that 
\begin{equation*}
\int_0^T R(s,x,h)^q \; \textrm{d}s \leq C_{T,q} h^{(2-q)/3}.
\end{equation*}
In the spirit of \cite{Zheng} once again, we control this integral by proving two separate estimates for it which we will then combine. These estimates are obtained by splitting $R$ up into different components. Note that 
\begin{equation*}
R(s,x,h)= R_1(s,x,h) + R_2(s,x,h) \geq 0,
\end{equation*}
where we define 
\begin{equation*}
R_1(s,x,h):=\frac{1}{\sqrt{s}} \left( 1- e^{-h^2/16s} \right),
\end{equation*}
and
\begin{equation*}\begin{split}
R_2(s,x,h):= \sqrt{s} \left[ \frac{1-e^{-x^2/4s}}{x^2} + \frac{1-e^{-(x+h)^2/4s}}{(x+h)^2} - 2e^{-h^2/16s} \left( \frac{1-e^{-\frac{x(x+h)}{4s}}}{x(x+h)} \right) \right].
\end{split}
\end{equation*}
Therefore, if $I(s,x,h) \geq 0$ such that $R_2(s,x,h) \leq I(s,x,h)$ we have that 
\begin{equation*}
R(s,x,h)^q \leq (R_1(s,x,h) + I(s,x,h))^q \leq C_q \left(R_1(s,x,h)^q + I(s,x,h)^q \right).
\end{equation*}
Our approach will therefore be to bound the integral of the $R_1$ and $I$ terms separately, and we do so for two different functions $I$. Starting with the $R_1$ term, we first note that 
\begin{equation*} \begin{split}
R_1(s,x,h) \leq \frac{1}{\sqrt{s}}, \; \; \textrm{ and that } \; \; 
R_1(s,x,h) \leq \frac{h^2}{16 s^{3/2}},
\end{split}
\end{equation*}
where we once again make use of the inequality $(1-e^{-x}) \leq x$. It then follows that 
\begin{equation*}\begin{split}
\int_0^T R_1(s,x,h)^q \textrm{d}s \leq \int_0^{h^2} \left(\frac{1}{\sqrt{s}} \right)^q \textrm{d}s + \int_{h^2}^{\infty} \frac{h^{2q}}{16^q s^{3q/2}} \textrm{d}s,
\end{split}
\end{equation*}
Calculating, we see that this is equal to $Ch^{2-q}$. For the $R_2$ term, we have the following bound corresponding to our first choice of $I$,
\begin{equation*}
R_2(s,x,h) \leq \sqrt{s} \left[ \frac{1}{2s}- 2e^{-h^2/16s} \left(\frac{1}{4s} - \frac{x(x+h)}{32 s^2}\right) \right], 
\end{equation*}
where we have made use of the fact that 
\begin{equation*}
x-\frac{x^2}{2} \leq 1-e^{-x} \leq x.
\end{equation*}
In this case, $I$ is then equal to 
\begin{equation*}
\frac{1}{2 \sqrt{s}}\left[ 1-e^{-h^2/16s} \right] +\frac{x(x+h)}{16 s^{3/2}}e^{-h^2/16s}.
\end{equation*}
The first of these terms is in the same form as $R_1$, and so we can control it in the same way. For the second term, we have that
\begin{equation*}
\begin{split}
\left( \frac{x(x+h)}{16 s^{3/2}}e^{-h^2/16s} \right)^q = C_q x^q(x+h)^q \left( \frac{e^{-(h^2)/16 s}}{s^{3/2}} \right)^q &  \leq C_q x^q(x+h)^q \left( \frac{s}{s^{3/2} h^2} \right)^q \\ & = C_q x^q(x+h)^q s^{-q/2} h^{-2q} ,
\end{split}
\end{equation*}
since $e^{-x} \leq \frac{1}{e x}$. This bound then gives that
\begin{equation*}
\begin{split}
\int_0^T \left( \frac{x(x+h)}{16 s^{3/2}}e^{-h^2/16s} \right)^q \;  \textrm{d}s & \leq \int_0^{h^2} \left( \frac{x(x+h)}{16 s^{3/2}}e^{-h^2/16s} \right)^q \;  \textrm{d}s + \int_{h^2}^{\infty} \left( \frac{x(x+h)}{16 s^{3/2}}e^{-h^2/16s} \right)^q \;  \textrm{d}s
\\ & \leq C_q x^q(x+h)^q \left( \int_0^{h^{2}} s^{-q/2} h^{-2q} \; \textrm{d}s + \int_{h^{2}}^{\infty} \frac{1}{s^{3q/2}} \textrm{d}s \right).
\end{split}
\end{equation*}
Calculating gives that the right hand side is equal to 
\begin{equation*}
C_q x^q (x+h)^q h^{2- 3q}.
\end{equation*}
So we obtain our first bound
\begin{equation}\label{firstbound}
\int_0^T R(s,x,h)^q \; \textrm{d}s \leq C_q \left[ h^{2-q} + x^q (x+h)^q h^{2- 3q} \right].
\end{equation}
We now prove our second estimate for the integral by bounding $R_2$ with a different choice of $I$. We have that 
\begin{equation}\label{J2}\begin{split}
R_2(s,x,h) \leq  & \sqrt{s} \left| \frac{1-e^{-x^2/4s}}{x^2} + \frac{1-e^{-(x+h)^2/4s}}{(x+h)^2} - 2 \left( \frac{1-e^{-\frac{x(x+h)}{4s}}}{x(x+h)} \right) \right| \\ &   + \sqrt{s} \left| 2\left(1- e^{-h^2/16s}\right)  \left( \frac{1-e^{-\frac{x(x+h)}{4s}}}{x(x+h)} \right)\right|.
\end{split}
\end{equation}
The second term on the right hand side is at most 
\begin{equation*}
\frac{C}{\sqrt{s}} \left(1- e^{-h^2/16s}\right).
\end{equation*}
This is once again in the form of $R_1$, and so can be controlled in the same way. In order to bound the first term on the right hand side of (\ref{J2}), we define here the function
\begin{equation*}
\phi(t,x):= \frac{1-e^{-x/4t}}{x}.
\end{equation*}
Then we can calculate that, for $x \in (0,1]$ and every $s >0$,
\begin{equation}\label{derbound}
\left|\frac{\partial \phi}{\partial x}(s,x) \right|= \left| -\frac{1}{x^2}(1-e^{-x/4s}) + \frac{1}{4xs} e^{-x/4s} \right| \leq \frac{C}{xs}. 
\end{equation}
Therefore, 
\begin{equation*}\begin{split}
\sqrt{s} & \left|  \frac{1-e^{-x^2/4s}}{x^2} + \frac{1-e^{-(x+h)^2/4s}}{(x+h)^2} - 2 \left( \frac{1-e^{-\frac{x(x+h)}{4s}}}{x(x+h)} \right) \right| \\  & = \sqrt{s}\left| \phi(s,x^2) + \phi(s,(x+h)^2) -2 \phi(s,x(x+h)) \right|.
\end{split}
\end{equation*}
Using (\ref{derbound}), we see that this is at most
\begin{equation*}
\frac{C}{x^2 \sqrt{s}} \times (xh + h^2) = C \left( \frac{h}{x\sqrt{s}} + \frac{h^2}{x^2 \sqrt{s}} \right).
\end{equation*}
We have that 
\begin{equation*}
C\int_0^T \left[\left( \frac{h}{x\sqrt{s}} + \frac{h^2}{x^2 \sqrt{s}} \right) \right]^q \textrm{d}s= C_{T,q} \left(\frac{h^q}{x^q} + \frac{h^{2q}}{x^{2q}} \right).
\end{equation*}
Putting these parts together, we have deduced our second estimate, 
\begin{equation*}\label{secondbound}
\int_0^T R(s,x,h)^q \textrm{d}s \leq C_{T,q} \left( h^{2-q} +\frac{h^q}{x^q} + \frac{h^{2q}}{x^{2q}} \right).
\end{equation*}
We can now conclude our proof by splitting into cases, and using the two inequalities in the different scenarios. For $x \leq h^{\frac{4q-2}{3q}}$, we use our first bound, (\ref{firstbound}), and obtain that 
\begin{equation*}\begin{split}
\int_0^T R(s,x,h)^q \; \textrm{d}s & \leq C_q \left[ h^{2-q} + x^q (x+h)^q h^{2- 3q} \right]\\ &  \leq C_q \left[ h^{2-q} + h^{\frac{8q-4 +6 -9q}{3}} \right] \leq C_q h^{\frac{2 -q}{3}},
\end{split}
\end{equation*}
where we have used that, since $q \in (1,2)$, we have $(4q-2)/3q \in (2/3,1)$ and so $h \leq h^{(4q-2)/3q}$. If $x > h^{\frac{4q-2}{3q}}$, we have by our second bound that 
\begin{equation*}\begin{split}
\int_0^T R(s,x,h)^q \textrm{d}s & \leq C_{T,q} \left( h^{2-q} +\frac{h^q}{x^q} + \frac{h^{2q}}{x^{2q}} \right) \\ & \leq C_{T,q} \left( h^{2-q} + h^{q- \frac{4q-2}{3}} + h^{2q-\frac{8q -4}{3}} \right) \\& =  C_{T,q} \left( h^{2-q} + h^{(2-q)/3} + h^{2(2-q)/3} \right) \\ & \leq C_{T,q} h^{(2-q)/3},
\end{split}
\end{equation*}
recalling that $h \in [0,1]$ and $q \in (1,2)$. This concludes the proof.
\end{proof}

\subsection{Proof of Inequality 4}

\begin{proof}
We begin by recalling the inequality. For $T>0$, $0 \leq s \leq t \leq T$ and $q \in (1,2)$, $\exists \; C_{T,q} >0$ such that 
\begin{equation*}
\sup\limits_{x \in [0,1]} \int_0^s \left[ \int_0^1 (\tilde{G}(t-r,x,y)- \tilde{G}(s-r,x,y))^2 \textrm{d}y \right]^q \textrm{d}r \leq C_{T,q}|t-s|^{(2-q)/2}.
\end{equation*}
Once again, we will prove the corresponding bound for $\tilde{J}$, given by (\ref{H}). The proof will have two steps. We will first perform a change of variables, which will allow us to bound the integral by $|t-s|^{(2-q)/2}$ multiplied by some integral which does not depend on $t,s$, which we denote by $I(x)$. We will then show that 
\begin{equation*}
\sup\limits_{x \geq 0} I(x) < \infty.
\end{equation*}
Let $k:= t-s$. We have that 
\begin{equation*}\begin{split}
\int_0^s &  \left[ \int_0^1 (\tilde{J}(s-r + k,x,y)- \tilde{J}(s-r,x,y))^2 \textrm{d}y \right]^q \textrm{d}r 
\\   & = \int_0^s \left[ \int_0^1 (\tilde{J}(r+ k,x,y)- \tilde{J}(r,x,y))^2 \textrm{d}y \right]^q \textrm{d}r 
\\ & = \int_0^s \left[ \int_0^1 \frac{1}{k}\left(\tilde{J}\left(\frac{r}{k}+ 1,\frac{x}{\sqrt{k}},\frac{y}{\sqrt{k}} \right)- \tilde{J}\left(\frac{r}{k},\frac{x}{\sqrt{k}},\frac{y}{\sqrt{k}}\right)\right)^2 \textrm{d}y \right]^q \textrm{d}r. 
\end{split}
\end{equation*}
Performing the change of variables $z:= y/\sqrt{k}$ and $u=r/k$, we obtain that this is at most $k^{(2-q)/2} \times I(\frac{x}{\sqrt{k}})$, where
\begin{equation*}
I(x)= \int_0^{\infty} \left[ \int_0^{\infty} \left( \tilde{J}(u+1,x,z)- \tilde{J}(u,x,z) \right)^2 \textrm{d}z \right]^q \textrm{d}u.
\end{equation*}
By inequality (1) of Lemma 3.2.1 in \cite{Zheng}, we have that there exists $C>0$ such that for $u \geq 0$
\begin{equation*}
\sup\limits_{x \geq 0} \int_0^{\infty} \tilde{J}(u,x,z)^2 \textrm{d}z \leq \frac{C}{\sqrt{u}}.
\end{equation*}
Therefore, for $u \geq 0$ we have  
\begin{equation*}
\sup\limits_{x \geq 0} \int_0^{\infty} \left( \tilde{J}(u+1,x,z)- \tilde{J}(u,x,z) \right)^2 \textrm{d}z \leq \frac{C}{\sqrt{u}}.
\end{equation*} 
It follows that, for $q \in (1,2)$, 
\begin{equation*}
\int_0^{1} \left[ \int_0^{\infty} \left( \tilde{J}(u+1,x,z)- \tilde{J}(u,x,z) \right)^2 \textrm{d}z \right]^q \textrm{d}u = C_{q} < \infty.
\end{equation*}
So it is only left to control
\begin{equation}\label{bound-}
\int_1^{\infty} \left[ \int_0^{\infty} \left( \tilde{J}(u+1,x,z)- \tilde{J}(u,x,z) \right)^2 \textrm{d}z \right]^q \textrm{d}u.
\end{equation}
We split the space integral here into two cases, when $\left\{ z \geq 2x \right\}$, and when $\left\{ z < 2x \right\}$. For the integral over the spatial region $\left\{ z< 2x \right\}$, we have that,
\begin{equation*}
\begin{split}
\int_1^{\infty} &  \left[ \int_0^{2x} \left( \tilde{J}(u+1,x,z)- \tilde{J}(u,x,z) \right)^2 \textrm{d}z \right]^q \textrm{d}u  \leq C_q \int_1^{\infty}  \left[ \int_0^{\infty} \left( J(u+1,x,z)- J(u,x,z) \right)^2 \textrm{d}z \right]^q  \textrm{d}u\\  \leq &  C_q \int_1^{\infty} \left[ \int_0^{\infty} \frac{1}{u} \left[ e^{-(x-z)^2/4(u+1)} - e^{-(x-z)^2/4u} \right]^2 \textrm{d}z \right]^q  \textrm{d}u\\ & +  C_q\int_1^{\infty} \left[ \int_0^{\infty} \frac{1}{u} \left[ e^{-(x+z)^2/4(u+1)} - e^{-(x+z)^2/4u} \right]^2 \textrm{d}z \right]^q \textrm{d}u \\ & + C_q\int_1^{\infty} \left[ \int_0^{\infty} \left( \frac{1}{\sqrt{u}} - \frac{1}{\sqrt{u+1}}\right)^2 \left( e^{-(x-z)^2 /4(u+1)} - e^{-(x+z)^2/4(u+1)} \right)^2 \textrm{d}z \right]^q \textrm{d}u \\ \leq & C_q\int_1^{\infty} \left[  \int_{\mathbb{R}} \frac{1}{u} \left( e^{-z^2/4(u+1)} - e^{-z^2/4u} \right)^2 \textrm{d}z \right]^q \textrm{d}u \\ & + C_q \int_1^{\infty} \frac{1}{u^{3q}} \left[  \int_0^{\infty} \left( e^{-(x-z)^2 /4(u+1)} - e^{-(x+z)^2/4(u+1)} \right)^2 \textrm{d}z \right]^q \textrm{d}u \\ \leq & C_q \int_1^{\infty} \left[ \int_{\mathbb{R}} \frac{1}{u} e^{-z^2/2(u+1)}\left( 1 - e^{-z^2(\frac{1}{4u} - \frac{1}{4(u+1)})} \right)^2 \textrm{d}z \right]^q  \textrm{d}u + C_q \int_1^{\infty} \frac{1}{u^{5q/2}} \textrm{d}u  \\ \leq &  C_q \int_1^{\infty} \left[  \int_{\mathbb{R}} \frac{z^4}{u^5}e^{-z^2/2(u+1)} \textrm{d}z \right]^q \textrm{d}u + C_q \int_1^{\infty} \frac{1}{u^{5q/2}} \textrm{d}u  = C_q \int_1^{\infty} \frac{1}{u^{5q/2}} \textrm{d}u.
\end{split}
\end{equation*}
This last term is finite and doesn't depend on $x$, so we have successfully bounded (\ref{bound-}) in the case when $\left\{z< 2x \right\}$. On the set $\left\{ z \geq 2x \right\}$, we have that 
\begin{equation*}
\begin{split}
\int_1^{\infty} & \left[ \int_{2x}^{\infty} \left( \tilde{J}(u+1,x,z)- \tilde{J}(u,x,z) \right)^2 \textrm{d}z \right]^q \textrm{d}u
\\ \leq C_q & \int_0^{\infty} \left[ \int_{2x}^{\infty} \frac{z^2}{x^2} \frac{1}{u}\left(  \left( e^{-(x-z)^2/4(u+1)} - e^{-(x+z)^2/4(u+1)} \right) -  \left( e^{-(x-z)^2/4u} - e^{-(x+z)^2/4u} \right)  \right)^2  \textrm{d}z \right]^q \textrm{d}u 
\\  & + C_q \int_0^{\infty} \left[ \int_{2x}^{\infty} \frac{z^2}{x^2} \left( \frac{1}{\sqrt{u}}- \frac{1}{\sqrt{u+1}} \right)^2 \left( e^{-(x-z)^2/4(u+1)} - e^{-(x+z)^2/4(u+1)} \right)^2 \textrm{d}z \right]^q \textrm{d}u 
\\ = C_q & \int_0^{\infty} \left[ \int_{2x}^{\infty} \frac{z^2}{x^2} \frac{1}{u}\left(  e^{-(x-z)^2/4(u+1)}\left(1 - e^{-xz/(u+1)} \right)  -  e^{-(x-z)^2/4u} \left( 1- e^{-xz/u} \right)  \right)^2  \textrm{d}z \right]^q \textrm{d}u 
\\  & + C_q \int_0^{\infty} \left[ \int_{2x}^{\infty} \frac{z^2}{x^2} \left( \frac{1}{\sqrt{u}}- \frac{1}{\sqrt{u+1}} \right)^2 \left( e^{-(x-z)^2/4(u+1)} - e^{-(x+z)^2/4(u+1)} \right)^2 \textrm{d}z \right]^q \textrm{d}u 
\\ \leq C_q & \int_1^{\infty} \left[ \int_{2x}^{\infty} \frac{z^2}{x^2} \frac{1}{u} \left( e^{-xz/(u+1)} - e^{-xz/u} \right)^2 e^{-(x-z)^2/2(u+1)} \textrm{d}z \right]^q \textrm{d}u 
\\ & +  C_q  \int_1^{\infty} \left[ \int_{2x}^{\infty} \frac{z^2}{x^2} \frac{1}{u} \left( 1 - e^{-xz/u} \right)^2 \left( e^{-(x-z)^2/4(u+1)} - e^{-(x-z)^2/4u} \right)^2  \textrm{d}z \right]^q \textrm{d}u 
 \\ & +  C_q  \int_1^{\infty} \left[ \int_{2x}^{\infty} \frac{z^2}{x^2} \frac{1}{u^3} e^{-(x-z)^2/2(u+1)} \left( 1- e^{-xz/(u+1)} \right)^2  \textrm{d}z \right]^q \textrm{d}u
\end{split}
\end{equation*}
This is at most 
\begin{equation*} \begin{split}
 C_q & \int_1^{\infty} \left[ \int_{2x}^{\infty} \frac{z^2}{x^2} \frac{1}{u}e^{-2xz/(u+1)} \left( 1- e^{-xz/u(u+1)} \right)^2 e^{-(x-z)^2/2(u+1)} \textrm{d}z \right]^q \textrm{d}u 
\\ & +  C_q  \int_1^{\infty} \left[ \int_{2x}^{\infty} \frac{z^4}{u^3} e^{-(x-z)^2/2(u+1)} \left( 1- e^{-(x-z)^2/4u(u+1)} \right)^2  \textrm{d}z \right]^q \textrm{d}u 
 \\ & +  C_q  \int_1^{\infty} \left[ \int_{2x}^{\infty} \frac{z^4}{u^3(u+1)^2} e^{-(z-x)^2/2(u+1)}   \textrm{d}z \right]^q \textrm{d}u
 \\ \leq &  C_q  \int_1^{\infty} \left[ \int_{0}^{\infty} \frac{z^4}{u^3(u+1)^2} e^{-z^2/8(u+1)}  \textrm{d}z \right]^q \textrm{d}u 
\\ & +  C_q  \int_1^{\infty} \left[ \int_{0}^{\infty} \frac{z^8}{u^5(u+1)^2}e^{-z^2/8(u+1)}  \textrm{d}z \right]^q \textrm{d}u 
 \\ & +  C_q  \int_1^{\infty} \left[ \int_{0}^{\infty} \frac{z^4}{u^3(u+1)^2} e^{-z^2/8(u+1)}  \textrm{d}z \right]^q \textrm{d}u
 \\ \leq & C_q \int_1^{\infty} \frac{1}{u^{5q/2}} \textrm{d}u = C_q < \infty.
\end{split} \end{equation*}

Therefore we have that $\sup\limits_{x \geq 0} I(x) < \infty$, concluding the proof.
\end{proof}
\subsection{Proof of Inequality 5}
\begin{proof}
We want to prove that, for $t \in (0,T]$
 \begin{equation*}
\sup\limits_{x \in [0,1]} \int_0^1 \left| \tilde{H}(t,x,y) \right| \textrm{d}y \leq \frac{C_T}{\sqrt{t}}.
\end{equation*}
We prove the corresponding bound for $K$, where we define, for $x \in (0,1]$, $y \in [0,1]$ and $t \geq 0$,
\begin{equation*}
K(t,x,y):= \frac{y}{x}\frac{\partial J}{\partial y}(t,x,y),
\end{equation*}
and for $y \in [0,1]$, $t \geq 0$ we define 
\begin{equation*}
K(t,0,y):= y \frac{\partial^2 J}{\partial x \partial t}(t,0,y).
\end{equation*}
Note that, for $t>0$, $x \in (0,1]$, $y \in [0,1]$ and $y \geq 2x$, 
\begin{equation*}
\begin{split}
\left| K(t,x,y) \right| \leq &  \frac{Cy}{t^{3/2}} \left[ e^{-(x-y)^2/4t} + e^{-(x+y)^2/4t} \right]+ \frac{Cy^2}{t^{3/2}x} \left[e^{-(x-y)^2/4t}-e^{-(x+y)^2/4t} \right] \\ \leq & \frac{Cy}{t^{3/2}}e^{-y^2/16t}  + \frac{Cy^2}{t^{3/2}x}e^{-(x-y)^2/4t}\left[ 1- e^{-xy/t} \right] \\ \leq & e^{-y^2/16t} \left[  \frac{Cy}{t^{3/2}}+ \frac{Cy^3}{t^{5/2}} \right].
\end{split}
\end{equation*}
By letting $x \downarrow 0$, we also have that the bound holds at $x=0$ for $t>0$ and $y \in [0,1]$. For $x \in (0,1]$, $y \in [0,1]$ and $y < 2x$, we have that 
\begin{equation*}
\begin{split}
|K(t,x,y)| \leq & 2 \left| \frac{\partial J}{\partial y}(t,x,y) \right| \\  \leq & \frac{C}{t^{3/2}} \left[ (x-y) e^{-(x-y)^2/4t} + (x+y)e^{-(x+y)^2/4t} \right] \\  \leq & \frac{C}{t^{3/2}} \left[ (x-y) e^{-(x-y)^2/16t} + (x+y)e^{-(x+y)^2/16t} \right]
\end{split}
\end{equation*}
It follows that 
\begin{equation*}
\sup\limits_{x \in [0,1]} \int_0^1 \left| K(t,x,y) \right| \textrm{d}y \leq C \int_{\mathbb{R}} e^{-y^2/16t} \left[  \frac{y}{t^{3/2}}+ \frac{y^3}{t^{5/2}} \right] \textrm{d}y \leq \frac{C}{\sqrt{t}}.
\end{equation*}
\end{proof}
\textbf{Acknowledgements.} 
The research of J. Kalsi was supported by EPSRC (EP/L015811/1).

\end{document}